\documentclass{article}

\usepackage{amssymb,amsmath,amsthm,amsfonts,enumerate}
\usepackage{color, graphicx,multicol}
\usepackage{url}

\usepackage{enumerate}
 
\usepackage{authblk}

\newcommand{\intR}{\int_{\mathbb R}}

\DeclareMathOperator{\supp}{supp}

\allowdisplaybreaks \numberwithin{equation}{section}

\theoremstyle{plain}
\newtheorem{thm}{Theorem}[section]
\newtheorem{cor}[thm]{Corollary}
\newtheorem{lem}[thm]{Lemma}
\newtheorem{prop}[thm]{Proposition}

\newtheorem{defn}[thm]{Definition}

\title{\bf A variational characterization of 2-soliton profiles for the KdV equation}

\author{John P. Albert}

\affil{Department of Mathematics, University of Oklahoma\\ Norman, OK 73019\\ \texttt{jalbert@ou.edu}}

\author{Nghiem V. Nguyen}

\affil{Department of Mathematics and Statistics, Utah State University\\   Logan, UT 84322   \\ \texttt{nghiem.nguyen@usu.edu}}

\begin{document}
\maketitle

\newcommand\blfootnote[1]{%
  \begingroup
  \renewcommand\thefootnote{}\footnote{#1}%
  \addtocounter{footnote}{-1}%
  \endgroup
}

\begin{abstract} It is well known that 2-soliton profiles for the
KdV equation are local minimizers of a constrained variational problem involving 
three polynomial conservation laws.  Here we show that 2-soliton profiles
are the global minimizers for this variational problem.  The proof, which proceeds via a profile decomposition,
also shows that every minimizing sequence converges strongly to the set of minimizing profiles.
\end{abstract}

\section{Introduction}
\label{sec:intro}
\renewcommand{\theequation}{\arabic{section}.\arabic{equation}}
\setcounter{section}{1} \setcounter{equation}{0}

The variational properties of multi-soliton solutions of the Kortweg-de Vries (KdV) equation have played a central role in the study of this equation since shortly after the discovery of its remarkable properties in the 1960s.  An early milestone was the paper \cite{L} of Lax, in which it is pointed out that multi-soliton profiles are critical points of constrained variational problems in which the constraint functionals and the objective functional are conserved under the flow defined by the KdV equation.  This suggested that it might be possible  to establish stability properties of multi-soliton solutions by using the conserved quantities as Lyapunov-type functionals.

\blfootnote{MSC Codes: 35Q53 35C08 35A15 \qquad Keywords:  multi-soliton, KdV, global minimizers}

 Benjamin \cite{Be} (see also Bona \cite{Bo}) took a first step in this direction by showing that solitary-wave profiles  are local minimizers in $H^1(\mathbb R)$ for the conserved functional $E_3$, subject to the constraint that $E_2$ be held constant (see Section \ref{sec:main} for the definition of $E_j$), and deducing the orbital stability of solitary-wave solutions as a consequence.   Later, the work of Cazenave and Lions (see for example, \cite{CL}, and the expository treatment in \cite{Al2}) established that solitary-wave profiles are actually global minimizers of this variational problem in a strong sense:  every minimizing sequence for the variational problem has a subsequence which, after appropriate translations, converges in $H^1(\mathbb R)$ to a solitary-wave profile. Orbital stability of solitary waves is an immediate consequence.

Maddocks and Sachs \cite{MS} generalized the theory of Benjamin and Bona to obtain a stability result for multi-soliton solutions of KdV.  A key step in their proof was to show that the profiles of $N$-soliton solutions are local minimizers in $H^N(\mathbb R)$ of the conserved functional $E_{N+2}$, when the functionals $E_2$, $E_3$, \dots, $E_{N+1}$ are held constant.  Their proof, like that of Benjamin and Bona for single solitons, did not yield information about global minimizers of the variational problem.

In this paper, we consider the special case $N=2$ of the variational problem considered in \cite{MS}:  that is, the problem of minimizing $E_4$ when $E_2$ and $E_3$ are held constant.   In our main result, Theorem \ref{mainthm} below, we show that indeed 2-soliton solutions represent the global minimizers for this variational problem.   

 An easy consequence of Theorem \ref{mainthm} is a stability result for
$2$-soliton solutions in $H^2(\mathbb R)$, stated below as Corollary \ref{stability}.   Of course, this is only a special case of the stability result of \cite{MS}, which was asserted for $N$-solitons for general $N$, not just for $N=2$.  Moreover, in recent years a number of papers have appeared on the topic of stability of multi-solitons which have improved on the result of \cite{MS}.  In particular, Killip and Visan \cite{KV} have proved a stability result for $N$-soliton solutions of KdV which is in some sense optimal:  it asserts stability in $H^{-1}$, or more generally in any space $H^s$ with $s \ge -1$.   Instead of the variational characterization of multi-solitons used here or in \cite{MS}, they use a different variational characterization, which is motivated by the inverse scattering theory for KdV, yet which is well-adapted to potentials in low-regularity Sobolev spaces where the classical inverse scattering theory does not apply.  We also note the recent work of Le Coz and Wang \cite{LCW}, who by building on and elucidating the work of \cite{MS} obtain a stability result for $N$-soliton solutions of the modified KdV equation.  Their methods should be transferrable to other integrable equations as well; and in particular it would be worth using them to revisit the stability theory for KdV multisolitons. 

 Apart from its consequences for stability theory, Theorem \ref{mainthm} is of interest in that it settles, at least in the case $N=2$, the question of whether multisolitons are actually global minimizers of a natural variational problem for KdV, expressed in terms of polynomial conservation laws which can be viewed as action variables in a formulation of KdV as an infinite-dimensional Hamiltonian system.  It can thus be viewed as a step towards obtaining an analogue for KdV on the real line of the elegant theory produced for the periodic KdV equation by Lax \cite{L3} and Novikov \cite{N}.  The proof has the advantage of simplicity:  it shows that the result is a straightforward consequence of the profile decomposition, a general phenomenon unconnected with the KdV equation or its integrable structure, once one shows that $2$-soliton profiles are minimizers for the constrained variational problem when consideration is restricted to the set of multi-soliton profiles.  In other words, the fact that $2$-soliton profiles are global minimizers in $H^2$ can be shown to follow from the profile decomposition, together with the fact that they are minimizers within the set of all multi-soliton profiles. 

An important caveat, however, is that the argument can only proceed because of the uniqueness result for $2$-solitons stated as Theorem \ref{unique} below; and such uniqueness results can be very difficult to prove in other settings.  In fact, one of the main reasons we have restricted ourselves to $2$-solitons in the present paper is that an analogue of Theorem \ref{unique} is not yet available for general $N$-solitons (see \cite{Al} for a discussion of what remains to be shown).   

The plan of the remainder of the paper is as follows.  In Section \ref{sec:main} we review some basic properties
 of $N$-soliton solutions and polynomial conservation laws for the KdV equation,  state our main result 
Theorem \ref{mainthm}, and sketch its proof.  In Section \ref{sec:reviewKdV}, we review
 the profile decomposition, following \cite{M}.   In Section \ref{sec:finitemin}, we analyze a finite-dimensional 
minimization problem which arises from restricting the admissible functions in \eqref{deflambda} and
 \eqref{defJ} to $N$-soliton profiles.  Section \ref{sec:proofmain} contains the proof of our main result, Theorem \ref{mainthm},
and concludes with  a proof of Corollary \ref{stability}.

{\it Note:} After this paper was written, we learned of a recent preprint of T. Laurens \cite{Lau}, which includes a complete solution of characterizing $N$-solitons, for general $N$, as global minimizers for the variational problems considered by Maddocks and Sachs.  In particular, our Theorem \ref{mainthm} is a special case of theorems proved in \cite{Lau}.  Furthermore, Laurens obtains general results on the behavior of minimizing sequences,  which not only recover Maddocks and Sachs' stability result for multisolitons,  but also partly corroborate and partly refute a conjecture we make below at the end of Section 2.  A key feature of the analysis of \cite{Lau} is that, rather than relying as we do on a uniqueness theorem for solutions of the Euler-Lagrange equation, Laurens proves uniqueness of global minimizers of the variational problem directly, using direct and inverse scattering theory (more specifically, using the Zakharov-Fadeev formulas for the conserved functionals $E_j$). 

\bigskip

\noindent {\it Notation.}

If $E$ is a measurable subset of $\mathbb R$ and $1 \le p < \infty$, we define $L^p(E)$ to be the space of Lebesgue measurable real-valued functions $u$ on $E$ such that 
$\|u\|_{L^p(E)}=\left(\int_E |u|^p\ dx\right)^{1/p}$ is finite.   In the case when $E = \mathbb R$, we sometimes denote $L^p(\mathbb R)$ by simply $L^p$, and denote the norm of $u$ in $L^2(\mathbb R)$ by $\|u\|_{L^2}$.  
  
When $E$ is an open set in $\mathbb R$, for $k \in \mathbb N$, we define the $L^2$-based Sobolev space $H^k=H^k(E)$ to be the space of all measurable functions on $E$ whose distributional derivatives up to order $k$ are in $L^2(E)$, with norm given by 
$$
\|u\|_{H^k(E)}=\left(\sum_{i=0}^k \int_E \left(\frac{d^iu}{dx^i}\right)^2\ dx\right)^{1/2} .
$$ 
Note that $H^0(E)=L^2(E)$.  In the case when $E = \mathbb R$, we sometimes denote $H^k(\mathbb R)$ by simply $H^k$, and denote the norm of $u$ in $H^k(\mathbb R)$ by $\|u\|_{H^k}$.

For $x \in \mathbb R$ and $r>0$ we denote by $B(x,r)$ the open ball in $\mathbb R$ centered at $x$ with radius $r$, or in other words the interval $(x-r,x+r)$.  Also, for any subset $E$ of $\mathbb R$, we denote by $\chi_E$ the characteristic function of $E$, so that $\chi_E(x)=1$ for $x \in E$ and $\chi_E(x)=0$ for $x \notin E$.

\section{Statement of main result} \label{sec:main}

\bigskip
 
We begin by reviewing the definition and some basic properties of $N$-soliton solutions of the Korteweg-de Vries (KdV) equation, and the associated sequence of polynomial conservation laws.   For more details and further references, the reader is referred to, for example, \cite{Di} and \cite{Di2}; the early papers \cite{L, L3, L2} of Lax are also very readable.

Suppose $N \in \mathbb N$, $0 < C_1 < \dots < C_N$ and $(\gamma_1,\dots, \gamma_N) \in \mathbb R^N$.
An {\it $N$-soliton profile function} is a function of the form
\begin{equation*}
\psi_{C_1,\dots,C_N;\gamma_1,\dots,\gamma_N}(x)=3(\Delta'/\Delta)'
\end{equation*}
where $\Delta$ is defined as the $N \times N$ determinant of Wronskian form,
\begin{equation*}
\Delta=\Delta(y_1,\dots,y_N) =\left|
\begin{matrix}
y_1 & \dots &y_N  \\
y_1' & \dots &y_N' \\
\dots & \dots & \dots \\
y_1^{(N-1)} & \dots & y_N^{(N-1)}  \\
\end{matrix}
\right|,
\end{equation*}
with
 \begin{equation*}
y_j(x)=e^{\sqrt{C_j}(x-\gamma_j)}+(-1)^{j-1}e^{-\sqrt{C_j}(x-\gamma_j)},  \quad j=1,\dots,N.
\end{equation*}
 
From $N$-soliton profile functions, we can construct {\it $N$-soliton solutions} of the KdV equation,
\begin{equation}
u_t +uu_x+ u_{xxx}=0,
\label{KdV}
\end{equation}
simply by defining
\begin{equation*}
u(x,t)=\psi_{C_1, \dots, C_N; \gamma_1(t), \dots,  \gamma_N(t)}(x)
\end{equation*}
where for $j=1,\dots, N$,
\begin{equation*}
\gamma_j(t)=a_j + C_j t,
\end{equation*}
and $(a_1, \dots, a_N) \in \mathbb R^N$ is arbitrary.

In particular, a single-soliton profile is obtained by taking $N=1$, in which case we have, for $C>0$ and $\gamma \in \mathbb R$,
\begin{equation*}
\psi_{C,\gamma}(x)= 3(y_1'/y_1)'
\end{equation*}
where
$y_1= e^{\sqrt C (x-\gamma)} + e^{-\sqrt C(x-\gamma)}$.
In other words,
\begin{equation*}
\psi_{C,\gamma}(x)=\frac{3C}{ \cosh^2(\sqrt{C}(x-\gamma))}.
\end{equation*}
Then a single-soliton solution of the KdV equation is obtained by taking
$$
u(x,t)=\psi_{C,\gamma(t)}(x)
$$
where $\gamma(t)=a+Ct$ and $a \in \mathbb R$ is arbitrary.  

If the constants $\gamma_1, \dots, \gamma_N$ are widely separated, then the profile function $\psi_{C_1,\dots,C_N;\gamma_1,\dots,\gamma_N}$
closely resembles a sum of single-soliton profiles $\sum_{i=1}^N \psi_{C_i,\gamma_i}$.   A particular instance of this well-known fact that we will use below is the following:

 \begin{lem}
 Suppose $0 < C_1 < C_2$; $\gamma_1, \gamma_2 \in \mathbb R$; and $\{x_n^1\}$ and $\{x_n^1\} $ are sequences such that 
 $$
 \lim_{n \to \infty}|x_n^1 - x_n^2| = \infty.
 $$
 Then
 $$
 \lim_{n \to \infty} \left\|  \psi_{C_1,\gamma_1 + x_n^1} + \psi_{C_2,\gamma_2 + x_n^2} - \psi_{C_1,C_2,\gamma_1+x_n^1,\gamma_2 + x_n^2} \right\|_{H^2(\mathbb R)} = 0.
 $$
 \label{sumoftwosols}
  \end{lem}

\begin{proof}
This follows immediately from Lemma 3.6 of \cite{ABN}.
\end{proof}
  
The variational problem we are concerned with here has, as objective and constraint functionals, polynomial conservation laws for the KdV equation.  These are functionals of the form 
$$
E_j(u)=\intR P_j\left(u,u_x,u_{xx}, \dots, \frac{\partial^{j-2} u}{\partial x^{j-2}}\right)\ dx
$$
where $j \ge 2$ and $P_j$ is a polynomial function of its arguments, and is of degree $j$ in $u$.  As shown, for example, in \cite{L2}, for each $j \ge 2$ there exists such a functional which is a conservation law for KdV, meaning that if $u(x,t)$ is a solution of \eqref{KdV}, then (at least formally),
$$
\frac{d}{dt}\left[ E_j(u(x,t))\right]=0
$$
for all $t \in \mathbb R$.  In this paper we are most concerned with the functionals $E_2$, $E_3$, and $E_4$, which are given by
\begin{equation*}
\begin{aligned}
E_2(u)&= \int_{\mathbb R} \frac12 u^2\ dx,\\
E_3(u)&=\int_{\mathbb R}\left(\frac12 u_x^2 - \frac16 u^3\right)\ dx,\\
E_4(u)&=\int_{\mathbb R} \left(\frac12 u_{xx}^2 -\frac56 uu_x^2 + \frac{5}{32} u^4\right)\ dx.
\end{aligned}
\end{equation*}
(There is also a conserved functional of degree 1 in $u$, given by $E_1(u) = \int_{\mathbb R} u\ dx$, which we do not use here.)

Sobolev embedding theorems imply that for each $N \ge 0$, $E_{N+2}$ defines a continuous functional on the Sobolev space $H^N=H^N(\mathbb R)$ of real-valued functions whose derivatives up to order $N$ are in $L^2(\mathbb R)$.   For $u \in H^N$, we denote by $\nabla E_k(u)$ the Fr\'echet derivative of $E_k$ at $u$, which coincides with the Gateaux derivative of $E_k$ and is therefore defined as a linear functional on  $H^N$ by 
$$
\nabla E_k(u)[v]= \lim_{\epsilon \to 0} \frac{E_k(u+\epsilon v)-E_k(u)}{\epsilon},
$$ 
and identified with a function $\nabla E_k(u)(x)$ in the usual way, so that $\nabla E_k(u)[v]=\int_{\mathbb R} \nabla E_k(u)(x) \cdot v(x)\ dx$ for $v \in H^N$.
In particular, the Fr\'echet derivatives of $E_2$, $E_3$, and $E_4$ are given by
\begin{equation*}
\begin{aligned}
\nabla E_2(u)&= u,\\
\nabla E_3(u) &= -u_{xx} -\frac12 u,\\
\nabla E_4(u) &= u_{xxxx} + \frac53 u_{xx}+\frac56(u_x)^2 + \frac58 u^3.
\end{aligned} 
\end{equation*} 

For fixed $0 < C_1 < \dots < C_N$,
define
\begin{equation}
S=S(C_1,\dots,C_N)=\left\{\psi_{C_1,\dots,C_N;\gamma_1,\dots,\gamma_N}(x): (\gamma_1,\dots,\gamma_N) \in \mathbb R^N\right\}.
\label{defS}
\end{equation}

It is well-known (for a proof, see for example Theorem 3.8  of \cite{Al}) that  there exist constants $\lambda_2$, \dots $\lambda_{N+1}$ such that each $\psi \in S$ satisfies the ordinary differential equation
\begin{equation}
\nabla E_{N+2}(\psi) = \sum_{k=2}^{N+1}\lambda_k \nabla E_k(\psi).
\label{NsolitonODE}
\end{equation}
Thus $N$-soliton profiles
 are (non-isolated) critical points for constrained variational problems involving the functionals $E_j$.  The family of ordinary differential equations \eqref{NsolitonODE} is collectively known as the stationary KdV hierarchy.  Due to work of  Novikov, Its/Matveev, and Gelfand/Dickey in the 1970's (see \cite{Al} for references), it is known that each equation in the hierarchy has the structure of a completely integrable Hamiltonian system, and indeed can be explicitly solved by integration.

In the case $N=2$, equation \eqref{NsolitonODE} takes the form
\begin{equation}
\nabla E_4(\psi) = \lambda_2 \nabla E_2(\psi) + \lambda_3 \nabla E_3(\psi),
\label{2solitonODE}
\end{equation} 
or 
\begin{equation*} 
\psi''''+\frac53\psi\psi''+\frac56(\psi')^2  +\frac58\psi^3=\lambda_2\psi - \lambda_3 (\psi'' + \frac12 \psi^2).
 \end{equation*}

We will make crucial use below of the fact that, for all $k \ge 2$, $E_k$ is constant on $S$, with its value on $S$ given by
\begin{equation}
E_k(C_1,\dots,C_N):=(-1)^k \frac{36}{2k-1}\sum_{j=1}^N C_j^{(2k-1)/2}.
\label{Ekvalue}
\end{equation}
To prove \eqref{Ekvalue}, one first shows  that it is valid in the case of a single-soliton profile, when $N=1$  (cf.\ equation (3.18) of \cite{MS}).   For a general $N$-soliton profile $\psi(x)=\psi_{C_1,\dots,C_N; \gamma_1,\dots, \gamma_N}(x)$, one then proves \eqref{Ekvalue} by considering a solution $u(x,t)$ of \eqref{KdV} with $u(x,0)=\psi(x)$.  Since $u(x,t)$ resolves into widely separated single-soliton profiles as $t \to \infty$, it follows that $\displaystyle \lim_{t \to \infty} E_k(u(x,t)) = E_k(C_1,\dots,C_N)$.  But since $E_k$ is a conserved functional for KdV, it then follows that $E_k(\psi)=E_k(C_1,\dots,C_N)$ as well.

 We now consider the constrained variational problem of minimizing the functional $E_4$ over $H^2(\mathbb R)$, subject to the constraints that $E_2$ and $E_3$ be held constant.   This is the same variational problem considered, in the case $N=2$, in the stability theory for $N$-solitons presented in \cite{MS}.  Whereas it was shown in \cite{MS} that $2$-soliton profiles are local minimizers of the variational problem, we will show in Theorem \ref{mainthm} that they are global minimizers. (Note that conversely, by the uniqueness result Theorem \ref{unique}, every global minimizer is a $2$-soliton profile.)  Moreover, every minimizing sequence for the problem must converge strongly in $H^2(\mathbb R)$ to the set of minimizers.

We begin by identifying the set $\Sigma$ of feasible constraints for the variational problem.

  \begin{prop}  Suppose $(a,b) \in \mathbb R^2$.   Then   there exists a nonzero function $g$  in $H^2(\mathbb R)$
  such that $E_2(g)=a$ and $E_3(g)=b$ if and only if $(a,b) \in \Sigma$, where
  $$
  \Sigma = \left\{(a,b) \in \mathbb R^2: \text{$a>0$ and $b \ge -\mu a^{5/3}$}\right\}, 
  $$
 and
 $$
 \mu=\frac{36}{5}\left(\frac{1}{12}\right)^{5/3}.
 $$  
\label{defsig}
  \end{prop}
  
  \begin{proof}   From Cazenave and Lions' variational characterization of solitary waves  (see for example Theorem 2.9 and Proposition 2.11 of \cite{Al2}), we know that for all $a>0$ and all $g \in H^1(\mathbb R)$ 
  such that $E_2(g)=a$, we have $E_3(g) \ge -\mu a^{5/3}$; and the minimum value $E_3(g)=-\mu a^{5/3}$ is attained when (and only when)
  $g=\psi_{C,\gamma}$, where $C=(a/12)^{2/3}$ and  $\gamma \in \mathbb R$ is arbitrary.   It follows that if $g$ is a nonzero function in $H^2(\mathbb R)$ with $E_2(g)=a$ and $E_3(g)=b$, then $(a,b) \in \Sigma$.
  
 Conversely, suppose $(a,b) \in \Sigma$.   For $\alpha > 0$, define  $g_\alpha(x)=\sqrt{\alpha}\psi_{C,\gamma}(\alpha x)$, where $C$ and $\gamma$ are as above. Then $E_2(g_\alpha)=a$ for every $\alpha>0$,
  $E_3(g_\alpha)=-\mu a^{5/3}$ when $\alpha = 1$, and $E_3(g_\alpha) \to +\infty$ as $\alpha \to \infty$.  Therefore, since $b \ge -\mu a^{5/3}$,  the intermediate value
     theorem implies the existence of some $\alpha \in [1,\infty)$ for which $E_3(g_{\alpha})=b$.  Thus by taking $g=g_{\alpha}$, we can satisfy $g \in H^2(\mathbb R)$,
     $E_2(g)=a$, and $E_3(g)=b$. 
\end{proof}  

For $(a,b) \in \Sigma$, define $\Lambda(a,b) \subseteq \mathbb R$ by
\begin{equation}
\Lambda(a,b)=\{E_4(g): \text{$g \in H^2(\mathbb R)$, $E_2(g)=a$, and $E_3(g)=b$}\}.
\label{deflambda}
\end{equation}
By the definition of $\Sigma$, the set on the right-hand side is nonempty, and we can therefore define 
\begin{equation}
J(a,b)= \inf \Lambda(a,b).
\label{defJ}
\end{equation}
(Notice that we do not exclude here the possibility that $J(a,b)=-\infty$.  However, as shown below at the beginning of Section \ref{sec:proofmain}, in fact $J(a,b) > - \infty$ for all $(a,b) \in \Sigma$.)

\begin{defn} Suppose $(a,b) \in \Sigma$.  We say that a function $\phi \in H^2(\mathbb R)$ is a {\rm minimizer for $J(a,b)$} if $E_2(\phi)=a$, $E_3(\phi)=b$, and $E_4(\phi)=J(a,b)$.
We say that a sequence $\left\{\phi_n\right\}$ of functions in $H^2(\mathbb R)$ is a {\rm minimizing sequence} for $J(a,b)$ if
$\displaystyle\lim_{n \to \infty} E_2(\phi_n) = a$, $\displaystyle\lim_{n \to \infty} E_3(\phi_n)=b$, and $\displaystyle\lim_{n\to \infty} E_4(\phi_n) = J(a,b)$.
\end{defn} 

A fact which is crucial for the proof of our main result below is that there is no choice of the numbers $\lambda_2$, $\lambda_3$ for which \eqref{2solitonODE} has any nontrivial solutions in $H^2(\mathbb R)$ besides $1$-soliton and $2$-soliton profiles:
 
\begin{thm}[\cite{Al}]
Suppose $\psi \in H^2$ is a nonzero solution of \eqref{2solitonODE}, in the sense of distributions, for some $\lambda_2, \lambda_3 \in \mathbb R$.  Then $\psi$ is either a 1-soliton or a 2-soliton profile for the KdV equation.  In other words we have
$\psi=\psi_{C_1, C_2; \gamma_1, \gamma_2}(x)$ for some real numbers $C_1$, $C_2$, $\gamma_1$, $\gamma_2$, with $0 \le C_1< C_2$.
\label{unique}
\end{thm}  
 
The proof given in \cite{Al} for Theorem \ref{unique} relies on the fact that, as mentioned above,  \eqref{2solitonODE} can be explicitly integrated.  

A corollary of Theorem \ref{unique} is that if minimizers for $J(a,b)$ exist, then they must of necessity be 1-soliton or 2-soliton profiles:

\begin{cor}  Suppose $a$ and $b$ are real numbers and  $u \in H^2$ is a nonzero minimizer for $J(a, b)$.   Then $u$ must be either a 1-soliton profile or a 2-soliton profile.   
\label{minsol}
\end{cor}

\begin{proof} 
According to Theorem 2 on page 188 of \cite{Lu}, if $u$ is a regular point of the constraint functionals
$E_2$ and $E_3$, meaning that the Fr\'echet derivatives $\nabla E_2(u)$ and $\nabla E_3(u)$ are linearly independent, then there must exist real numbers $\lambda_2$ and $\lambda_3$ such that the equation $\nabla E_4(u) = \lambda_2 \nabla E_2(u) + \lambda_3 \nabla E_3(u)$ holds,
at least in the sense of distributions.  In this case, by Theorem \ref{unique}, $u$ is either a 1-soliton or a 2-soliton profile.  On the other hand, if $u$ is not a regular point of the constrained functionals, then since $u$ is nonzero and therefore $\nabla E_2(u) \ne 0$, it must be the case that $\nabla E_3(u)= \lambda \nabla E_2(u)$ for some $\lambda \in \mathbb R$.    But it is an elementary exercise (see for example Theorem 4.2 of \cite{Al}) to show that the only possible nontrivial solutions of $\nabla E_3(u)= \lambda \nabla E_2(u)$ in $H^2$ are 1-soliton profiles.  
\end{proof} 

The preceding result of course leaves open the question of whether any 1-soliton or 2-soliton profiles are in fact minimizers for $J(a,b)$.  Our main result determines the set of values of $(a,b)$ within $\Sigma$ for which minimizers for $J(a,b)$ exist; 
and for such values of $(a,b)$ determines the value of $J(a,b)$, describes all the minimizers for $J(a,b)$,
and describes the behavior of minimizing sequences for $J(a,b)$.
If $S \subseteq H^2(\mathbb R)$, we say that a sequence $\{\phi_n\}$ converges to $S$ in $H^2(\mathbb R)$ norm if
\begin{equation*}
d(\phi_n,S)= \inf_{\psi \in S} \|\phi_n - \psi\|_{H^2(\mathbb R)} \to 0 \quad \text{as $n \to \infty$,}
\end{equation*}
or, equivalently, if there exists a sequence $\{\psi_n\}$ of elements of $S$ such that
$$
\lim_{n \to \infty}\|\phi_n -\psi_n\|_{H^2(\mathbb R)} = 0.
$$ 

\begin{thm}  Suppose $(a,b) \in \Sigma$; that is, $a > 0$ and $b \ge -\mu a^{5/3}$, where $\mu = \frac{36}{5}\left(\frac{1}{12}\right)^{5/3}$.
 
\begin{enumerate} 
 
\item  
 If 
 \begin{equation}
 b=-\mu a^{5/3},
 \label{abeq}
 \end{equation}
  then every minimizing sequence for $J(a,b)$ converges to $S(C)$ in $H^2(\mathbb R)$ norm, where $S(C)$ is
 as defined in \eqref{defS}, with $C = (a/12)^{2/3}= (-5b/36)^{2/5}$.
 Every element of $S(C)$ is a minimizer for $J(a,b)$, and $J(a,b)=E_4(C)$.  
 
 \item 
 If 
 \begin{equation}
 \frac{-\mu a^{5/3}}{2^{2/3}}  > b >  -\mu a^{5/3},
 \label{abineq1}
 \end{equation}
 then every minimizing sequence for $J(a,b)$ converges to $S(C_1,C_2)$ in $H^2(\mathbb R)$ norm,
 where $(C_1,C_2)$ is the unique pair of numbers such that $0 < C_1 < C_2$, $E_2(C_1,C_2)=a$, and $E_3(C_1,C_2)=b$.
 Every element of $S(C_1,C_2)$ is a minimizer for $J(a,b)$, and $J(a,b)=E_4(C_1,C_2)$.

\item
 If 
\begin{equation}
 b \ge -\frac{\mu a^{5/3}}{2^{2/3}},
\label{abineq2}
\end{equation}
 then there do not exist any minimizers for $J(a,b)$ in $H^2$.
 
 \end{enumerate}
 \label{mainthm}
 \end{thm}

An immediate consequence of Theorem \ref{mainthm} is a stability result for 2-soliton solutions of the KdV equation. 
 This recovers (by a different proof) the special case $N=2$ of the general result for $N$-soliton solutions 
given by Maddocks and Sachs in \cite{MS}:
 
\begin{cor}  Suppose  $0 < C_1 < C_2$.  Then every minimizing sequence $\left\{\phi_n\right\}$ for
 $J(E_2(C_1,C_2),E_3(C_1,C_2))$ converges to $S(C_1,C_2)$ in $H^2(\mathbb R)$ norm.  
 Moreover,  $S=S(C_1,C_2)$ is stable, in the sense that for every $\epsilon > 0$ there exists $\delta > 0$ such that
if $u_0 \in H^2(\mathbb R)$ and $d(u_0,S) < \delta$, then $d(u(\cdot,t),S) < \epsilon$ for all $t > 0$.
\label{stability}
\end{cor}
 
{\it Remark.} It follows from this stability result that there are $C^1$ functions $\gamma_1(t)$ 
and $\gamma_2(t)$ defined for $t \ge 0$ such that
\begin{equation*}
\|u(\cdot,t)-\psi_{C_1,C_2;\gamma_1(t),\gamma_2(t)}\|_{H^2(\mathbb R)} < \epsilon
\end{equation*}
and $|\gamma_1'(t)-C_1|< \epsilon$, $|\gamma_2'(t)-C_2|< \epsilon$ for all $t>0$.  See \cite{ABN}.

\bigskip

We now sketch the proof of Theorem \ref{mainthm}.  A detailed proof will be given in the sections which follow.  

  Suppose $(a,b) \in \Sigma$ with 
 $\displaystyle -\mu a^{5/3} \le b < \frac{-\mu a^{5/3}}{2^{2/3}}$, and let    $\{\phi_n\}$ be an arbitrary minimizing sequence for $J(a,b)$.  Note that (since subsequences of minimizing sequences for $J(a,b)$ are again minimizing sequences), to prove part 1 or part 2 of the Theorem, it is enough to show that every minimizing sequence for $J(a,b)$ has a subsequence which exhibits the desired convergence.   Therefore, in what follows we can freely pass to subsequences of $\{\phi_n\}$, when necessary.
 
We can apply the profile decomposition, in the form of Corollary \ref{combinecases},  to the sequence $\{\rho_n\}$ defined by
 $$\rho_n= |\phi_n|^2 + |\phi_n'|^2+|\phi_n''|^2.$$
 It follows that, by passing to a subsequence (which we continue to denote by $\{\phi_n\}$), we can decompose $\phi_n$ as  
\begin{equation*}
\phi_n = \sum_{i=1}^n v^i_n + w_n,
\end{equation*}
where  the sequence $\{w_n\}$
 is vanishing in the sense of Definition \ref{defvanish}, and for each $i \in \mathbb N$ the sequence $\{v^i_n\}_{n \in \mathbb N}$ concentrates around some sequence
$\{x_n^i\}_{n \in \mathbb N}$, in the sense of Definition \ref{defconcentrate}.  From the concentration property
of $\{v_n^i\}_{n \in \mathbb N}$ and the fact that $\{\phi_n\}$ is a minimizing sequence,
 it follows that the sequence $\{v^i_n\}_{n \in \mathbb N}$ can be suitably translated so that it converges, 
weakly in $H^2(\mathbb R)$ and strongly in $H^1(\mathbb R)$, to a minimizer $g_i$ for the variational problem
\begin{equation*}
E_4(g_i)= \inf \left\{E_4(\phi): \phi \in H^2(\mathbb R), E_2(\phi)=a_i, E_3(\phi)=b_i\right\}
\end{equation*}
for some real numbers $a_i$, $b_i$.  As a critical point of this constrained variational problem, $\psi = g_i$ 
must satisfy the Euler-Lagrange equation \eqref{2solitonODE}. 
 From Theorem  \ref{unique} it then follows that for each $i$ we have $g_i =\psi_{D_{1i},D_{2i}}$ for
 some numbers $D_{1i}$ and $D_{2i}$ with $0\le D_{1i}< D_{2i}$.

In parts 1 and 2 of Theorem \ref{mainthm}, our assumption on $(a,b)$ implies that there exist numbers $C_1$ and $C_2$ with $0 \le C_1 < C_2$ such that 
$E_2(\psi_{C_1,C_2})=a$ and $E_3(\psi_{C_1,C_2})=b$.   
Therefore, by definition of $J(a,b)$, we have that
$$
J(a,b) \le E_4(\psi_{C_1,C_2}) = (36/7)\left(C_1^{7/2}+C_2^{7/2}\right).
$$ 
From the profile decomposition and the fact that $\{\phi_n\}$ is a minimizing sequence, we can obtain that 
\begin{equation*}
\begin{aligned}
\sum_{i=1}^\infty E_2(g_i)&= 36\sum_{i=1}^\infty\left(D_{1i}^3+D_{2i}^3\right)\le \lim_{n \to \infty}E_2(\phi_n)= 36\left(C_1^3 + C_2^3\right)\\
\sum_{i=1}^\infty E_3(g_i)& = \frac{-36}{5}\sum_{i=1}^\infty\left(D_{1i}^5+D_{2i}^5\right)\le \lim_{n \to \infty}E_3(\phi_n)= \frac{-36}{5} \left(C_1^5 + C_2^5\right)\\
\sum_{i=1}^\infty E_4(g_i)& = \frac{36}{7} \sum_{i=1}^\infty\left(D_{1i}^7+D_{2i}^7\right)\le \lim_{n \to \infty} E_4(\phi_n) \le \frac{36}{7}\left(C_1^7 + C_2^7\right).\\
\end{aligned}
\end{equation*} 

Permuting the terms of the sequence $(D_{11}^{1/2}, D_{21}^{1/2}, D_{21}^{1/2}, D_{22}^{1/2}, D_{31}^{1/2}, D_{32}^{1/2}, \dots)$ so that they form a
decreasing sequence $(x_1,x_2, x_3, \dots)$; and defining $y_1 = C_2^{1/2}$ and $y_2=C_1^{1/2}$, we thus have that
\begin{equation*}
\begin{aligned}
\sum_{i=1}^\infty x_i^3 &\le y_1^3 + y_2^3\\
\sum_{i=1}^\infty x_i^5 &\ge y_1^5 + y_2^5\\
\sum_{i=1}^\infty x_i^7 &\le y_1^7 + y_2^7.
\end{aligned}
\end{equation*}
We analyze this system of inequalities in Section \ref{sec:finitemin}, where we show (cf.\ Lemma \ref{zy7lem} and Lemma \ref{ruleoutthreeposinf})
 that it can only be satisfied if  $x_1=y_1$, $x_2=y_2$, and $x_i=0$ for all $i \ge 3$. This can be interpreted as
 saying that, among the $N$-soliton profiles of the KdV equation, the only ones which could possibly solve the
 variational problem are $1$-soliton profiles (in the case when $C_1=0$) and $2$-soliton profiles (in the case when $C_1 > 0$).

This information, combined with the control on the functions $\{v_n^i\}$ and $w_n$ afforded by the profile decomposition, is
enough to allow us to deduce that the functions in the minimizing sequence $\{\phi_n\}$ are either of the form
\begin{equation*}
\phi_n(x)=\psi_{C_1,C_2}(x+x_n)+ r_n(x)
\end{equation*}
for some sequence $\{x_n\}$ of real numbers, where $r_n \to 0$ in $H^2(\mathbb R)$; or of the form
\begin{equation*}
\phi_n(x)=\phi_{C_1}(x+x^1_{n}) + \phi_{C_2}(x+x^2_{n}) + r_n(x)
\end{equation*}
for some pair of sequences $\{x_n^1\}$ and $\{x_n^2\}$ of real numbers with $|x^1_n -x^2_n| \to \infty$, 
where again  $r_n \to 0$ in $H^2(\mathbb R)$.  In either case, this shows that the set of minimizing functions for the variational problem
consists of the set $S(C_1,C_2)$, and that $\{\phi_n\}$ converges to $S(C_1,C_2)$ in $H^2(\mathbb R)$ norm.

Part 3 of Theorem \ref{mainthm} will follow from a simpler argument:  under the given assumptions on $(a,b)$, no 1-soliton profile or
2-soliton profile $\psi$ can exist satisfying the constraints $E_2(\psi)=a$ and $E_3(\psi)=b$.  But any minimizer for the variational problem
must satisfy the associated Euler-Lagrange equation \eqref{2solitonODE}, and therefore by Theorem \ref{unique}
must be either a 1-soliton profile or a 2-soliton profile.  Hence no minimizers can exist.  

\bigskip

We conclude this section with a couple of remarks on possible extensions of Theorem  \ref{mainthm}.

The analysis given below of the behavior of minimizing sequences under the assumptions \eqref{abeq} or \eqref{abineq1}
should also be applicable in case \eqref{abineq2} holds, and suggests that in the latter case, minimizing sequences $\{\phi_n\}$
should, as $n \to \infty$, come to resemble superpositions of widely separated single-soliton profiles.   Thus, for example, if
$\displaystyle b=-\frac{ma^{5/3}}{2^{2/3}}$ we expect that if $\{\phi_n\}$ is a minimizing sequence for $J(a,b)$, then there will exist a number $C>0$ and 
 sequences $\{\gamma_{1n}\}$ and $\{\gamma_{2n}\}$ with $\displaystyle \lim_{n \to \infty} |\gamma_{1n}-\gamma_{2n}| =\infty$
such that $$\lim_{n \to \infty}\|\phi_n - (\psi_{C,\gamma_{1n}}+\psi_{C,\gamma_{2n}})\|_{H^2} = 0.$$   Similarly, if
$\displaystyle b>-\frac{ma^{5/3}}{2^{2/3}}$,  we expect that the functions in a typical minimizing sequence for $J(a,b)$ would resemble
a superposition of three or more single-soliton profiles, two of which have equal amplitudes and whose distance from each
 other increases to infinity as $n \to \infty$.   We do not pursue this topic further here, however.

Also, the method of proof of Theorem \ref{mainthm} should apply as well to the variational problems satisfied by $N$-soliton profiles
for general $N \in \mathbb N$.  One important obstacle we have encountered, however, 
 is that of proving an analogue of Theorem \ref{unique}:  i.e., of showing
that for all possible choices of the numbers $\lambda_2, \dots, \lambda_{N+1}$, the Euler-Lagrange equation 
\begin{equation}
\nabla E_{N+2}(\psi) = \lambda_2 \nabla E_2(\psi) + \dots +\lambda_{N+1} \nabla E_{N+1}(\psi)
\label{ELN}
\end{equation}
 has no solutions in $H^N$ besides $N$-soliton profiles.  As noted in \cite{Al}, the explicit integration of equation \eqref{ELN} can be carried out for general $N$
just as it can for $N=2$, whenever the value of $(\lambda_2,\dots,\lambda_{N+1})$ corresponds to that of an $N$-soliton profile; 
 but technical difficulties arise in proving that solutions corresponding to other values of $(\lambda_2,\dots,\lambda_{N+1})$ are singular.\footnote{As noted above in the introduction, a generalization of Theorem \ref{mainthm} for $N$-soliton profiles for general $N \in \mathbb N$ has recently been accomplished by Laurens in \cite{Lau}.}

\section{Profile decomposition} \label{sec:reviewKdV}

The proof of Theorem \ref{mainthm} uses an elaboration of the method of concentration compactness, known as
 {\it profile decomposition}, which details the ways in which a sequence of measures of bounded total mass can lose compactness.
The technique dates back at least to  \cite{G} and \cite{S}, and in some form even earlier (see for
example \cite{BC} and \cite{St}).  Many generalizations, elaborations, and applications of profile decomposition
to different settings have since appeared: for a modern account, see the monograph \cite{T}.  Here we use a version 
which is due to Mari\c s \cite{M}, valid for arbitrary bounded sequences of Borel measures on any metric space;
mainly because its general setting seemed easily adaptable to our specific context. (Actually, for simplicity 
of notation we here restate Mari\c s' results in slightly less general terms than in \cite{M}, restricting here to consideration
of bounded sequences of nonnegative functions in $L^1(\mathbb R)$.)  It is likely, however, that using a version already
tailored to Sobolev spaces (such as is used in \cite{Lau}) would shorten the proof below.

\begin{defn}  We say that a sequence $\{f_n\}$ of nonnegative functions in $L^1(\mathbb R)$ is {\rm vanishing} if for every $r>0$, we have
$$
\lim_{n \to \infty} \sup_{y\in \mathbb R} \int_{B(y,r)} f_n = 0.
$$
\label{defvanish}
\end{defn}

\begin{defn} If $\{f_n\}$ is a sequence of nonnegative functions in $L^1(\mathbb R)$ and $\{x_n\}$ is a sequence of real numbers, we say
 that $\{f_n\}$ {\rm concentrates around} $\{x_n\}$ if for every $\epsilon > 0$, there exists $r_\epsilon > 0$ such that
$$
\int_{\mathbb R \backslash B(x_n, r_\epsilon)} f_n < \epsilon
$$
for every $n \in \mathbb N$.
\label{defconcentrate}
\end{defn}

\begin{thm}[\cite{M}]  Suppose $\{\rho_n\}$ is a  sequence of nonnegative functions which is bounded in $L^1(\mathbb R)$.  
Then either $\{\rho_n\}$ is vanishing, or there exists a subsequence of $\left\{\rho_n\right\}$ (which we continue to denote by 
$\left\{\rho_n\right\}$), which satisfies one of the following two properties:  either

\begin{enumerate}[(1)]
\item there exist $k \in \mathbb N$ and for each $i \in \{1,\dots, k\}$  a number $m_i > 0$ and sequence of balls 
$\{B(x^i_n,r^i_n)\}_{n \in \mathbb N}$ in $\mathbb R$  with $\displaystyle \lim_{n \to \infty} r^i_n =\infty$,  such that

\begin{enumerate}
\item $B(x^i_n,r^i_n) \cap B(x^j_n,r^j_n) = \emptyset$ for all $n \in \mathbb N$ and all $i, j \in \{1,\dots,k\}$ with $i \ne j$,

\item for each $i \in \{1,\dots, k\}$, $\displaystyle \lim_{n \to \infty} \int_{B(x^i_n,r^i_n/2)}\rho_n = m_i$,

\item  for each $i \in \{1,\dots, k\}$, $\displaystyle \lim_{n \to \infty}\int_{B(x^i_n,r^i_n)\backslash B(x^i_n,r^i_n/2)}\rho_n = 0$,

\item for each $i \in \{1,\dots,k\}$, the sequence $\displaystyle \{\rho_n \chi_{B(x^i_n, r^i_n)}\}_{n \in \mathbb N}$ concentrates around $\{x^i_n\}_{n \in \mathbb N}$, 

\item the sequence $\displaystyle\{\rho_n \chi_{\mathbb R\backslash \cup_{i=1}^k B(x^i_n,r^i_n)}\}_{n \in \mathbb N}$ is vanishing;
\end{enumerate}

or

\item for each $i \in \mathbb N$ there is a number $m_i > 0$ and a sequence of balls $\displaystyle \{B(x^i_n,r^i_n)\}_{n = i, i+1, i+2,\dots}$ in $\mathbb R$, with $\displaystyle \lim_{n \to \infty} r^i_n =\infty$,  such that
\begin{enumerate}

\item $B(x^i_n,r^i_n) \cap B(x^j_n,r^j_n) = \emptyset$ for all $i, j \in \mathbb N$ with $i \ne j$, and all $n \in \mathbb N$ with $n \ge i$ and $n \ge j$,

\item for each $i \in \mathbb N$, $\displaystyle \lim_{n \to \infty} \int_{B(x^i_n,r^i_n/2)}\rho_n = m_i$,

\item for each $i \in \mathbb N$, $\displaystyle \sum_{n=i}^\infty \int_{B(x^i_n,r^i_n)\backslash B(x^i_n,r^i_n/2)}\rho_n  \le \frac{1}{2^i}$,

\item  for each $i \in \mathbb N$, the sequence $\displaystyle \{\rho_n \chi_{B(x^i_n,r^i_n)}\}_{n \ge i}$ concentrates around $\{x^i_n\}_{n \ge i}$,

\item  the sequence $\displaystyle \{\rho_n \chi_{\mathbb R\backslash \cup_{i=1}^n B(x^i_n,r^i_n)}\}_{n \in \mathbb N}$ is vanishing, and

\item if for each $N \in \mathbb N$ and each $n \ge N$, we define $\displaystyle g^N_n=\rho_n\chi_{\mathbb R \backslash \cup_{i=1}^NB(x^i_n,r^i_n)}$, and define the increasing
function $q^N_n(r)$ for $r>0$ by
$$
q^N_n(r) = \sup_{y \in \mathbb R}\int_{B(y,r)} g^N_n,
$$
then
\begin{equation}
\lim_{N \to \infty}\left(\lim_{r \to \infty}\left(\limsup_{n \to \infty} q^N_n(r)\right)\right) = 0.
\label{closertovanishing}
 \end{equation}
\end{enumerate}
\end{enumerate}
\label{profdecomp}
\end{thm}
 
One can view \eqref{closertovanishing} as saying that although, for any given value of $N$, 
the sequence $\{g^N_n\}_{n \in \mathbb N}$ is not necessarily vanishing,
 it does come closer, in some sense, to being a vanishing sequence as $N \to \infty$.

To shorten our proof of Theorem \ref{mainthm}, we observe that the two cases in Theorem \ref{profdecomp} can be combined into one, if we drop the requirement that $m_i > 0$ for each $i$:

\begin{cor}
 Suppose $\{\rho_n\}$ is a  sequence of nonnegative functions which is bounded in $L^1(\mathbb R)$, and suppose $\{\rho_n\}$ is not vanishing.  
Then there exists a subsequence of $\left\{\rho_n\right\}$ (which we continue to denote by $\left\{\rho_n\right\}$),
 a sequence $\{m_i\}_{i \in \mathbb N}$ of nonnegative numbers,  
and for each $i \in \mathbb N$ a sequence of balls $\displaystyle \{B(x^i_n,r^i_n)\}_{n \in \mathbb N}$ in $\mathbb R$, 
with $\displaystyle \lim_{n \to \infty} r^i_n =\infty$,  such that
\begin{enumerate}[(a)]

\item $B(x^i_n,r^i_n) \cap B(x^j_n,r^j_n) = \emptyset$ for all $i, j \in \mathbb N$ with $i \ne j$, and all $n \in \mathbb N$,

\item for each $i \in \mathbb N$, $\displaystyle \lim_{n \to \infty} \int_{B(x^i_n,r^i_n/2)}\rho_n = m_i$,

\item for each $i \in \mathbb N$, $\displaystyle \sum_{n=i}^\infty \int_{B(x^i_n,r^i_n)\backslash B(x^i_n,r^i_n/2)}\rho_n  \le \frac{1}{2^i}$,

\item  for each $i \in \mathbb N$, the sequence $\displaystyle \{\rho_n \chi_{B(x^i_n,r^i_n)}\}_{n \in \mathbb N}$ concentrates around $\{x^i_n\}_{n \in \mathbb N}$,

\item  the sequence $\displaystyle \{\rho_n \chi_{\mathbb R\backslash \cup_{i=1}^n B(x^i_n,r^i_n)}\}_{n \in \mathbb N}$ is vanishing, and

\item if for each $N \in \mathbb N$ and each $n \in \mathbb N$, we define $\displaystyle g^N_n=\rho_n\chi_{\mathbb R \backslash \cup_{i=1}^NB(x^i_n,r^i_n)}$,
 and define the increasing function $q^N_n(r)$ for $r>0$ by
$$
q^N_n(r) = \sup_{y \in \mathbb R}\int_{B(y,r)} g^N_n,
$$
then
\begin{equation*}
\lim_{N \to \infty}\left(\lim_{r \to \infty}\left(\limsup_{n \to \infty} q^N_n(r)\right)\right) = 0.
\end{equation*} 
\end{enumerate} 
\label{combinecases}
\end{cor}
 
\begin{proof}To obtain Corollary \ref{combinecases} from Theorem \ref{profdecomp}, we observe
 that if (2) holds in Theorem \ref{profdecomp}, then 
the statements in Corollary \ref{combinecases} will also hold if we simply define 
$B(x_n^i,r_n^i) = \emptyset$ when $i < n$.  So we need only consider the case when  
(1) holds in Theorem \ref{profdecomp}.

Define $E_n = \mathbb R \backslash \cup_{i=1}^k B(x_n^i, r_n^i)$ for $n \in \mathbb N$.  
Since $\left\{\rho_n \chi_{E_n}\right\}_{n \in \mathbb N}$ is vanishing, then for each fixed $j \in \mathbb N$,
$$
\lim_{n \to \infty} \sup_{y \in \mathbb R}\int_{B(y,j)} \rho_n \chi_{E_n}=0.
$$
Therefore we can define a sequence $n_1 < n_2 < n_3 < \dots$ such that
$$
\sup_{y \in \mathbb R} \int_{B(y,j)} \rho_{n_j} \chi_{E_{n_j}} \le \frac{1}{2^{j+1}}
$$
for all $j \in \mathbb N$.  If we now pass to the subsequence $\{\rho_{n_j}\}_{j \in \mathbb N}$,
 continuing to denote this subsequence by $\{\rho_n\}_{n \in \mathbb N}$, 
we have that
\begin{equation}
\sup_{y \in \mathbb R} \int_{B(y,n)} \rho_{n} \chi_{E_n} \le \frac{1}{2^{n+1}}
\label{case1tocase2}
\end{equation}
for all $n \in \mathbb N$.  Also, because of part (1)(c) of Theorem \ref{profdecomp}, by passing to a further subsequence we can guarantee
that
\begin{equation}
\int_{B(x^i_n,r^i_n)\backslash B(x^i_n,r^i_n/2)}\rho_n \le \frac{1}{2^{n+1}}
\label{case1tocase2iltk}
\end{equation}
holds for all $i \in \{1,2,\dots,k\}$ and all $n \in \mathbb N$ as well.

For each $i \ge k+1$, we set $m_i = 0$, and for all $n \in \mathbb N$ we define   $r_n^i = n$ if $n \ge i$ and $r_n^i = 0$ if $n < i$.  
For each fixed $n \in \mathbb N$, we define a sequence $\{x_n^j\}$ inductively for all $i \ge k+1$
 by choosing $x_n^i$ to be any real number such that
$B(x_n^i,r_n^i)$ is disjoint from $\cup_{j=1}^{i-1} B(x_n^j,r_n^j)$.  Then we have that $\displaystyle \lim_{n \to \infty} r_n^i = \infty$ for each
$i \in \mathbb N$, and part (a) of the Corollary holds.  

For all $i \ge k+1$ and for all $n \in \mathbb N$, since $B(x_n^i,r_n^i) \subset E_n$, it follows from \eqref{case1tocase2} that
\begin{equation}
\int_{B(x_n^i,r_n^i)} \rho_n \le \frac{1}{2^{n+1}}.
\label{igtkballtozero}
\end{equation}
This implies that parts (b) and (c) of the Corollary hold for all $i \ge k+1$.  For $1 \le i \le k$ we already know that part (b) holds, and part (c)
follows from \eqref{case1tocase2iltk}.

To prove part (d), we fix $i$ such that $i \ge k+1$, and observe that by \eqref{igtkballtozero},
for every $\epsilon > 0$ we can find $N \in \mathbb N$ such that $\displaystyle \intR \rho_n \chi_{B(x_n^i,r_n^i)} < \epsilon$ for all $n > N$.
Also, for each $n \in \{1, \dots, N\}$, since $\rho_n \chi_{B(x_n^i,r_n^i)} \in L^1(\mathbb R)$, we can find $r_{\epsilon, n} >0$ such that
$$
  \int_{\mathbb R\backslash B(x_n^i,r_{\epsilon, n})} \rho_n \chi_{B(x_n^i,r_n^i)} < \epsilon.
$$
   Then if we set $r_\epsilon = \max \{r_{\epsilon,1},\dots, r_{\epsilon, N}\}$, we have that
$$
 \int_{\mathbb R\backslash B(x_n^i,r_\epsilon)} \rho_n \chi_{B(x_n^i,r_n^i)} < \epsilon
$$
for all $n \in \mathbb N$.  This proves part (d) for all $i$ such that $i \ge k+1$, and we already know that part (d) holds for $1 \le i \le k$.

Finally, part (e) of the Corollary follows immediately from part (1)(e) of Theorem \ref{profdecomp}, as does part (f) of the Corollary; since
(1)(e) of Theorem \ref{profdecomp} implies that $\displaystyle \lim_{n \to \infty} q_n^N(r)=0$ for every $r > 0$ and every $N \ge k$. 
\end{proof}

We record the following important feature of vanishing sequences.

\begin{lem}  Suppose $2 < p \le \infty$.   Then there exists a constant $C_p > 0$ such that for all $u \in H^1(\mathbb R)$,
\begin{equation}
\|u\|_{L^p(\mathbb R)} \le C_p \left(\sup_{y \in \mathbb R} \int_{B(y,1)}|u'|^2 + |u|^2 \ dx \right)^{\frac12-\frac{1}{p}} \|u\|_{H^1(\mathbb R)}^{\frac{2}{p}}
\end{equation}
\label{lpvanish}
\end{lem}

\begin{proof}  This lemma is standard; a proof can be found, for example, in \cite{M}.
\end{proof}

\begin{cor}  Suppose $\{g_n\}$ is a bounded sequence in $H^1(\mathbb R)$.   If $\{|g_n|^2 + |g_n'|^2\}$ is vanishing, then $\displaystyle \lim_{n \to \infty}\|g_n\|_{L^p} = 0$ for all $p>2$.
\end{cor}  
 
 In the remainder of this section, we prove some lemmas regarding the behavior of the profile decomposition for bounded sequences in $H^2(\mathbb R)$.  
 
 Suppose $\{\phi_n\}$ is an arbitrary bounded sequence in $H^2(\mathbb R)$, and let $B<\infty$ be such that
 \begin{equation}
 \|\phi_n\|_{H^2} \le B 
\label{defB}
 \end{equation}
 for all $n \in \mathbb N$.  Define  
\begin{equation}
\rho_n :=\phi_n^2 + (\phi_n')^2 + (\phi_n'')^2.
\label{defrhon}
\end{equation}
  Since $\{\phi_n\}$ is bounded in $H^2$, then $\{\rho_n\}$ is bounded in $L^1$, and we can apply Corollary \ref{combinecases} to $\{\rho_n\}$. 
  
  For the remainder of this section we assume:
  
  \begin{equation}
  \text{$\{\rho_n\}$ is not a vanishing sequence.}
  \label{assumenovanish}
  \end{equation}
 
 Since $\{\rho_n\}$ does not vanish, then from Corollary \ref{combinecases}  we obtain
a sequence of balls $\{B(x_n^i,r_n^i)\}_{n \in \mathbb N}$ for each $i \in \mathbb N$, satisfying properties (a) to (f).

 Define $\eta$ to be a smooth function on $\mathbb R$ such that $\eta(x)=1$ 
for $|x| \le 1/2$ and $\eta(x)=0$ for $|x| \ge 1$; and for $R > 0$, define $\eta_R(x)=\eta(x/R)$.  
For each $i \in \mathbb N$ and $n \in \mathbb N$, define
\begin{equation}
\label{defvni}
v^i_n(x)=\phi_n(x)\eta_{r^i_n}(x-x^i_n).
\end{equation} 
We then have the decomposition
 \begin{equation}
\phi_n = \sum_{i=1}^n v^i_n+w_n,
\label{phidecomp}
\end{equation}
where for each $n \in \mathbb N$ we define 
\begin{equation}
\label{defwn}
w_n(x) = \phi_n(x)\widetilde \eta_n(x), 
\end{equation} 
with
$$
\widetilde \eta_n(x)= 1 - \sum_{i=1}^n \eta_{r^i_n}(x - x^i_n).
$$

 For all $n \in \mathbb N$ and $i \in \mathbb N$, define
 \begin{equation}
 \begin{aligned}
A^i_n &= B(x^i_n,r^i_n)\backslash B(x^i_n,r^i_n/2)\\
Z^i_n &= \mathbb R \backslash B(x^i_n,r^i_n/2)\\ 
 W_n &= \mathbb R \backslash \cup_{i=1}^n B(x_n^i,r_n^i).
\end{aligned}
\label{defAin}
\end{equation} 
 Then 
\begin{equation}
\begin{aligned}
\supp v^i_n &\subseteq B(x^i_n,r^i_n)\\
\supp w_n &\subseteq W_n \cup\left( \cup_{i=1}^n A^i_n\right)\\
\supp w_n' &\subseteq \cup_{i=1}^n A^i_n.
\end{aligned} 
\label{supports}
\end{equation}   
  
\begin{lem} Let $\{\phi_n\}$ be an arbitrary bounded sequence of functions in $H^2(\mathbb R)$, and let $B$, $\eta$, $\rho_n$, $v_n^i$, $w_n$, and $A^i_n$ be as defined above, under the assumption that \eqref{assumenovanish} holds.  Then there exists a constant $C>0$ 
depending only on $\eta$ and $B$ (and in particular, not on $m$ or $n$) such that for all $n \in \mathbb N$, and for $m=2,3,4$, 
\begin{equation} 
\left|E_m(\phi_n) - \sum_{i=1}^n E_m(v^i_n)-E_m(w_n)\right| \le C \sum_{i=1}^n\int_{A^i_n} \rho_n.
 \label{hksplit2}
\end{equation} 
\label{hksplitlem2}
\end{lem}

\begin{proof} We substitute \eqref{phidecomp} into $E_2(\phi_n)= \frac12 \intR \phi_n^2$, and expand,
  expressing $E_2(\phi_n)$ as a sum of integrals.    Since $v^i_n$ and $v^j_n$ have disjoint supports 
for $i \ne j$, all integrals in this expression whose integrands contain two factors of $v^i_n$ 
with distinct values of $i$ will vanish.  We thus obtain
$$
E_2(\phi_n) = \sum_{i=1}^n E_2(v^i_n) + E_2(w_n) + \sum_{i=1}^n \intR v^i_n w_n.
$$
 Since the intersection of the supports of $v^i_n$ and $w_n$ is contained in $A^i_n$,
 and $|v^i_n| \le |\phi_n|$ and $|w_n| \le |\phi_n|$ everywhere on $\mathbb R$, we have
\begin{equation}
\intR |v^i_n w_n| = \int_{A^i_n}|v^i_n w_n| \le \int_{A^i_n} |\phi_n|^2 \le \int_{A^i_n} \rho_n.
\label{vinwn}
\end{equation}
This then establishes \eqref{hksplit2} for $m=2$.

Substituting \eqref{phidecomp} into the expression for $E_3(\phi_n)$, we obtain the estimate
\begin{equation*}
\begin{aligned}
&\left|E_3(\phi_n)-\sum_{i=1}^n E_3(v^i_n) - E_3(w_n)\right| \le \\
&\sum_{i=1}^n\left(\intR |v^i_n{}'||w_n'| + \intR|v^i_n|^2| w_n| + \intR|v^i_n||w_n|^2\right).
\end{aligned} 
\label{e31est}
\end{equation*} 
Then we write
\begin{equation}
\begin{aligned}
\intR |v^i_n{}'|| w_n'| &= \int_{A^i_n}|v^i_n{}'|| w_n'| \le C \int_{A^i_n}\left( |\phi_n|^2 + |\phi_n'|^2\right)\le C\int_{A^i_n} \rho_n,\\
 \intR |v^i_n|^2|w_n| &\le \|v_n\|_{L^\infty}\int_{A^i_n}|v^i_n w_n| \le C\|\phi_n\|_{L^\infty}\int_{A^i_n}  \rho_n \le C \int_{A^i_n}  \rho_n,
 \end{aligned}
 \label{estimatesforE3}
 \end{equation}
and similarly for $\intR|v^i_n||w_n|^2$.  (Here and in what follows we use $C$ to stand for various constants which depend only on $\eta$ and $B$.)  
This establishes \eqref{hksplit2} for $m=3$.

Finally, to prove \eqref{hksplit2} for $m=4$, we substitute \eqref{phidecomp} into $E_4(\phi_n)$ and write
\begin{equation}
\begin{aligned}
&\left|E_4(\phi_n)-\sum_{i=1}^n E_4(v^i_n) - E_4(w_n)\right| \le \\
&C\intR\sum_{i=1}^n \left(|v^i_n{}''w_n''| + |v^i_nv^i_n{}' w_n'| + |v^i_n||w_n'|^2+|v^i_n{}'|^2|w_n|+
\sum_{\substack{s,t \ge 1\\s+t \le 4}}|v^i_n|^s|w_n|^t\right).
\end{aligned} 
\label{e41est}
\end{equation}  
All the terms on the right side of \eqref{e41est} can be estimated like the terms in the preceding paragraphs.   For example, we have
\begin{equation*}
\begin{aligned}
\intR |v^i_n|^3|w_n| &\le \|v_n\|_{L^\infty}^2\int_{A^i_n}|v^i_n w_n| \le C\|\phi_n\|_{L^\infty}^2\int_{A^i_n}  \rho_n \le C \|\phi_n\|_{H^1}^2\int_{A^i_n}  \rho_n\\
\intR |v^i_n{}'|^2|w_n| &\le \|w_n\|_{L^\infty}\int_{A^i_n}|v^i_n{}'|^2 \le C\|\phi_n\|_{L^\infty}\int_{A^i_n}  \rho_n \le C \|\phi_n\|_{H^1}\int_{A^i_n}  \rho_n.
\end{aligned}
\end{equation*} 
Clearly, similar estimates hold for the remaining integrals; we omit the details.
\end{proof}

\begin{lem}  In addition to the assumptions of Lemma \ref{hksplitlem2}, assume that  $f \in H^2$ with $\|f\|_{H^2} \le B$.
  For each  $n \in \mathbb N$ and each  $i \in \{1,\dots,n\}$, define
\begin{equation*}
\tilde \phi_n = f + \sum_{\substack{ j=1\\
                j \ne i}}^n v^j_n + w_n.
\end{equation*}  
Then  there exists a constant $C>0$, depending only on $\eta$ and $B$, such that 
 \begin{equation}
   \label{e2subst2}
 \left|E_m(\tilde\phi_n) -E_m(f)-   \sum_{\substack{ j=1\\
                j \ne i}}^n E_m(v^j_n)-E_m(w_n)\right|\le C \sum_{\substack{ j=1\\
                j \ne i}}^n \int_{A^j_n} \rho_n +C\|f\|_{H^{m-2}(Z^i_n)},
 \end{equation}  
 for $m=2$, $3$, and $4$.
\label{hksubstlem2}
\end{lem}

\begin{proof} Proceeding as in the proof of Lemma \ref{hksplitlem2}, we write
\begin{equation*}
\begin{aligned}
 &\left|E_2(\tilde\phi_n) -E_2(f)-   \sum_{\substack{ j=1\\
                j \ne i}}^n E_2(v^j_n)-E_2(w_n)\right|\\
               & \le C\intR|fw_n|+C\sum_{\substack{j=1\\j \ne i}}^n
\left( \intR|fv^j_n|+\intR|v^j_nw_n| \right)  
\end{aligned}
\end{equation*}
 The  sum $\displaystyle \sum_{\substack{j=1\\j \ne i}}^n\intR |v^j_nw_n|$ is estimated as in \eqref{vinwn}. 
  Also, since the support of $w_n$ is contained in $Z^i_n$, and the same is true of the support of $v^j_n$ whenever $j \ne i$, then
$$
\intR |fw_n| \le \|f\|_{L^2(Z^i_n)}\|w_n\|_{L^2} \le C\|f\|_{L^2(Z^i_n)}\|\phi_n\|_{L^2} \le C \|f\|_{L^2(Z^i_n)}.
$$ 
Finally, since the supports of the functions $\{v^j_n\}_{j \in \mathbb N}$ are mutually disjoint,
 and for $j \ne i$ are all contained in $Z^i_n$, there exists $C$ depending only on $\eta$ and $B$ (and not on $k$, $n$ or $i$) such that
$$
\sum_{\substack{j=1\\j \ne i}}^n \intR|fv^j_n| \le C\int_{Z^i_n}|f\phi_n| \le C \|f\|_{L^2(Z^i_n)}.
$$
   This proves  \eqref{e2subst2} for $m=2$.

Similarly, \eqref{e2subst2} is proved for $m=3$ by expanding $E_3(\phi_n)$ as in Lemma \ref{hksplitlem2},
 then using estimates such as \eqref{estimatesforE3}, together with the estimates 
$$
\intR|f'|\left(\sum_{j \ne i}|v^j_n{}'|+|w_n'|\right) \le C\|f'\|_{L^2(Z^i_n)}\|\phi_n\|_{H^1}\le C\|f\|_{H^1(Z^i_n)}\|\phi_n\|_{H^1},
$$
$$
\intR|f|^2\left(\sum_{j \ne i}|v^j_n|+|w_n|\right) \le \|f\|_{L^\infty}\|f\|_{L^2(Z^i_n)}\|\phi_n\|_{L^2} \le C\|f\|_{L^2(Z^i_n)},
$$
and a similar estimate for $\intR |f|(\sum_{j \ne i}|v^j_n|^2+|w_n|^2)$.

Finally, \eqref{e2subst2} is proved for $m=4$ by expanding $E_4(\phi_n)$ to obtain an expression similar to \eqref{e41est},
 but with additional terms on the right-hand side of the form
\begin{equation*}
\begin{aligned}
&\sum_{j \ne i}\intR\left(|f''||v^j_n{}''|+ |f||v^j_n{}'|^2 +|f|^2|v^j_n{}'|+|f||f'||v^j_n{}'|\right)+\\
&+\intR\left( |f''||w_n''|  + |f||w_n'|^2+ |f|^2|w_n'| +|f||f'||w_n'|\right)\\
&+\sum_{\substack{s,t \ge 1\\s+t \le 4}}\intR |f|^s\left(\sum_{j \ne i}|v^j_n|+|w_n|\right)^t.
 \end{aligned}
\end{equation*}
  These can each be estimated by the terms on the right-hand side of \eqref{e2subst2}.   For example, we have
$$
\intR|f|^3|w_n| \le \|f\|_{L^\infty}\|w_n\|_{L^\infty}\int_{Z^i_n}|f|^2 
\le C\|f\|_{H^1}^2\|w_n\|_{H^1}\|f\|_{L^2(Z^i_n)}
$$
and
$$
\intR|f||f'||w_n'| \le \|f\|_{L^\infty}\|w_n'\|_{L^2}\left(\int_{Z^i_n}|f'|^2\right)^{1/2} 
\le C\|f\|_{H^1}\|w_n\|_{H^1}\|f\|_{H^1(Z^i_n)}.
$$
The remaining terms are estimated similarly.  We omit the details, which are straightforward.
\end{proof}

\begin{lem} Under the same assumptions as in Lemma \ref{hksplitlem2}, we have
\begin{equation}
\lim_{n \to \infty} \sum_{i=1}^n\int_{A_n^i} \rho_n = 0.
\end{equation}
\label{sumrhontozero}
\end{lem}

\begin{proof}  For given $\epsilon > 0$, choose $N_1 \in \mathbb N$ so that 
$\displaystyle \sum_{i=N_1}^\infty \frac{1}{2^i}< \frac{\epsilon}{2}$. If $n \ge N_1$ then
 we can use part (c) of Corollary \ref{combinecases} to write
  $$
  \sum_{i=N_1}^n \int_{A_n^i}\rho_n \le \sum_{i=N_1}^n \sum_{j=i}^\infty \int_{A^i_j} \rho_j \le \sum_{i=N_1}^n\frac{1}{2^i} < \frac{\epsilon}{2}.
  $$  
But we also have from part (c) of Corollary \ref{combinecases} that for each
 fixed $i \in \mathbb N$, $\displaystyle \lim_{n \to \infty}\int_{A_n^i}\rho_n = 0$.
  Therefore, once $N_1$ has been chosen, we can find $N_2 \in \mathbb N$ so that if $n \ge N_2$ then
  $\displaystyle \sum_{i=1}^{N_1 - 1} \int_{A_n^i}\rho_n < \frac{\epsilon}{2}$.  
It follows that for $n \ge \max(N_1,N_2)$ we have
  $\displaystyle \sum_{i=1}^n \int_{A_n^i} \rho_n < \epsilon$, as desired.  
\end{proof}

 \begin{lem}
 Under the same assumptions as in Lemma \ref{hksplitlem2}, we have
 $$\lim_{n \to \infty} \|w_n\|_{L^p(\mathbb R)} = 0$$
  for all $p$ such that $2 < p \le \infty$.  
\label{lplimit2}
\end{lem}  

\begin{proof}  
For each $n \in \mathbb N$, we have from \eqref{supports}  that 
$w_n = w_n\left(\chi_{W_n}+\sum_{i=1}^n \chi_{A^i_n}\right)$ and
 $w_n' =  w_n' \left(\chi_{W_n}+\sum_{i=1}^n \chi_{A^i_n}\right)$;
 and from the definition of $w_n$ we see that there exists
a constant $C>0$ such that for every $n \in \mathbb N$ and $x \in \mathbb R$,
 $w_n^2(x) \le C \phi_n^2(x) \le C \rho_n(x)$, and
$(w_n')^2(x) \le C\left(\phi_n^2(x) + (\phi_n'(x))^2\right) \le C \rho_n(x)$.  Therefore,  we can write
\begin{equation}
\begin{aligned}
\sup_{y \in \mathbb R}\int_{B(y,1)}& \left(w_n^2 +(w_n')^2\right)\\
&=\sup_{y \in \mathbb R}\left[ \int_{B(y,1)} \left(w_n^2+(w_n')^2\right)\chi_{W_n} 
+ \sum_{i=1}^n\int_{B(y,1)}\left(w^2_n+(w_n')^2\right) \chi_{A^i_n}\right]\\
&\le C\sup_{y \in \mathbb R}\int_{B(y,1)}\rho_n\chi_{W_n} + C\sum_{i=1}^n\int_{A^i_n}\rho_n.
\end{aligned}
\label{wnvanish2}
\end{equation} 
But 
\begin{equation}
\lim_{n \to \infty} \sup_{y \in \mathbb R} \int_{B(y,1)}\rho_n\chi_{W_n} = 0,
\label{rhonvanishesonwn}
\end{equation}
 by part (e) of Corollary \ref{combinecases}.  Combining Lemma \ref{sumrhontozero}, 
 \eqref{wnvanish2}, and \eqref{rhonvanishesonwn} gives
$$
\lim_{n \to \infty} \sup_{y \in \mathbb R}\int_{B(y,1)}\left(w_n^2 + (w_n')^2\right)\ dx = 0.
$$
Since $\{w_n\}$, like $\{\phi_n\}$, is a bounded sequence in $H^1(\mathbb R)$,
 the proof is then completed by applying Lemma \ref{lpvanish}.
\end{proof}

\begin{lem} Under the same assumptions as in Lemma \ref{hksplitlem2}, we have
\begin{equation}
\lim_{N \to \infty} \limsup_{n \to \infty} \sum_{i=N}^n \intR |v_n^i|^3 = 0
\label{limlimsupf3}
\end{equation}
and
\begin{equation}
\lim_{N \to \infty} \limsup_{n \to \infty} \sum_{i=N}^n \intR\left[ \left|v_n^i \left( {v_n^i}'\right)^2\right|+  |v_n^i|^4\right] = 0.
\label{limlimsupf4}
\end{equation}
\end{lem}

\begin{proof} 
For a fixed value of $n \in \mathbb N$, the supports of the functions $\{v^i_n\}_{i=1,\dots,n}$ are mutually disjoint.  Therefore, if we define  
$$ 
f_{N,n} =\sum_{i=N}^n v^i_n,
$$
 we can write, for all $N,n$ such that $N < n$,
\begin{equation*}
\sum_{i=N}^n \intR \left|v^i_n\right|^3 = \intR \left|f_{N,n}\right|^3.
\end{equation*} 
Now
\begin{equation}
\begin{aligned}
\intR \left|f_{N,n}\right|^3 &\le C \|f_{N,n}\|_{H^1(\mathbb R)}^2\left( \sup_{y \in \mathbb R} \int_{B(y,1)} 
\left(|f_{N,n}'|^2 + |f_{N,n}|^2\right)\right)^{1/2}\\
& \le C \left(\sup_{y \in \mathbb R} \int_{B(y,1)} \rho_n \chi_{\mathbb R \backslash \cup_{i=1}^N B(x^i_n,r^i_n)}\right)^{1/2}
=C \left(q^N_n(1)\right)^{1/2},
\end{aligned}
\label{fNnest}
\end{equation}
where $q^N_n(r)$ is the function defined in part (f) of Corollary \ref{combinecases}. 
 (In obtaining \eqref{fNnest}, we used Lemma \ref{lpvanish} along with the facts that 
the support of $f_{N,n}$ lies outside $\cup_{i=1}^N B(x^i_n,r^i_n)$,
 and that $|f_{N,n}|^2$ and $|f_{N,n}'|^2$ are majorized pointwise by $C \rho_n$, 
where $C$ depends only on the cutoff function $\eta$.)   Since $q^N_n(r)$ is an
 increasing function of $r$, it follows from part (f) of Corollary \ref{combinecases} that 
\begin{equation}
 \lim_{N \to \infty} \limsup_{n \to \infty} q^N_n(1)=0.
\label{limlimsup}
\end{equation}
Therefore \eqref{limlimsupf3} follows from \eqref{fNnest}.

Similarly, we can use Lemma \ref{lpvanish} to write
\begin{equation} 
\begin{aligned}
  \sum_{i=N}^n \intR  |v_n^i|^4&=
\intR (f_{N,n})^4 \le C \|f_{N,n}\|_{H^1(\mathbb R)}^2
\left( \sup_{y \in \mathbb R} \int_{B(y,1)} \left(|f_{N,n}'|^2 + |f_{N,n}|^2\right)\right) \\
&\le C q^N_n(1),
\end{aligned}
\label{f4}
\end{equation}
and by H\"older's inequality, Sobolev embedding, and \eqref{fNnest}, we have
\begin{equation}
\begin{aligned}
 \sum_{i=N}^n \intR \left|v_n^i \left( {v_n^i}'\right)^2\right| & = \intR \left|f_{N,n}(f_{N,n}')^2\right|
 \le \left(\intR  \left|f_{N,n}'\right|^{8/3} 
 \right)^{3/4}\left(\intR\left|f_{N,n}\right|^3     \right)^{1/3} \\
&\le C \|f_{N,n}\|_{H^1(\mathbb R)}^{8/3} \left(q_n^N(1)\right)^{1/6}.
\end{aligned}
\label{ffp2}
\end{equation}
Estimates \eqref{f4} and \eqref{ffp2} together with \eqref{limlimsup} then imply \eqref{limlimsupf4}.
\end{proof}

Fix $i \in \mathbb N$, and for $n \in \mathbb N$ define 
\begin{equation}
\theta^i_n(x)=v^i_n(x+x^i_n).
\label{defthetain}
\end{equation}
  Since $\{\theta^i_n\}_{n \in \mathbb N}$
 is a bounded sequence in $H^2(\mathbb R)$, then by passing to a subsequence, we may assume that it converges weakly
 in $H^2(\mathbb R)$.  In fact, by a diagonalization argument,  we may replace 
$\{\phi_n\}_{n \in \mathbb N}$ by a subsequence with the property that we may assume  that  $\{\theta_n^i\}_{n \in \mathbb N}$ converges weakly in $H^2(\mathbb R)$ for every $i \in \mathbb N$.  
  (In what follows we will often replace sequences by subsequences without changing notation.)  Define, for each $i \in \mathbb N$,
 \begin{equation}
 g_i = \text{weak limit in $H^2(\mathbb R)$ of $\{\theta^i_n\}_{n \in \mathbb N}$}.
 \label{defgi}
 \end{equation}    Further,
 by again passing to appropriate subsequences, we may assume that the sequences
 $\{E_2(\theta^i_n)\}_{n \in \mathbb N}$ and $\{E_3(\theta^i_n)\}_{n \in \mathbb N}$ converge.  Define
\begin{equation}
\begin{aligned}
 a_i &=\lim_{n \to \infty} E_2(\theta^i_n) = \lim_{n \to \infty} E_2(v^i_n)\\
 b_i &=\lim_{n \to \infty} E_3(\theta^i_n) = \lim_{n \to \infty} E_3(v^i_n).
\end{aligned}
\label{defaibi}
\end{equation}

\begin{lem}   Under the same assumptions as in Lemma \ref{hksplitlem2}, define $\{\theta^i_n\}$, $g_i$, $a_i$, and $b_i$ as above.  Then for each $i \in \mathbb N$,  $\{\theta^i_n\}_{n \in \mathbb N}$ converges to $g_i$ in the norm of $H^1(\mathbb R)$.  
\label{giconvstronginf}
\end{lem}

\begin{proof} First note that  
by part (d) of Corollary \ref{combinecases}, for every $\epsilon > 0$ there exists $R_\epsilon > 0$ such that 
$$
\int_{B(x^i_n,r^i_n)\backslash B(x^i_n,R_\epsilon)}\rho_n  <\epsilon^2
$$
for all $n \in \mathbb N$. By taking $R_\epsilon$ larger if necessary, we may assume as well that
$$
\int_{\mathbb R \backslash B(0,R_\epsilon)}\left( (g_i'')^2 + (g_i')^2 + g_i^2\right) < \epsilon^2.
$$  
From the definition of $\theta^i_n$, it is easy to see that there exists a constant $C$ such that  for all $n$,
$$
\|\theta^i_n\|^2_{H^2(\mathbb R \backslash B(0,R_\epsilon))} \le C 
\int_{B(x^i_n,r^i_n)\backslash B(x^i_n,R_\epsilon)}\rho_n,
$$
and therefore
$$
\|\theta^i_n - g_i\|_{H^2(\mathbb R \backslash B(0,R_\epsilon))} \le \|\theta^i_n\|_{H^2(\mathbb R \backslash B(0,R_\epsilon))} +
 \|g_i\|_{H^2(\mathbb R \backslash B(0,R_\epsilon))} < (\sqrt C +1)\epsilon.
$$ 
On the other hand, since the inclusion of $H^2(B(0,R_\epsilon))$ into $H^1(B(0,R_\epsilon))$ is compact, then 
$\{\theta^i_n\}_{n \in \mathbb N}$  has a subsequence $\{\theta^i_{n_k}\}_{k \in \mathbb N}$ that
 converges strongly to $g_i$ in $H^1(B(0,R_{\epsilon}))$.  Then for all sufficiently large $k$,
 $\|\theta^i_{n_k} - g_i\|_{H^1(B(0,R_\epsilon))}< \epsilon$, and therefore $\|\theta^i_{n_k} - g_i\|_{H^1(\mathbb R)} < (\sqrt C + 2)\epsilon$.   

It follows from the preceding that for every $\epsilon > 0$, there exists a subsequence $\{\theta^i_{n_k}\}_{k \in \mathbb N}$ of $\{\theta^i_n\}_{n \in \mathbb N}$
 such that   $\|\theta^i_{n_k} - g_i\|_{H^1(\mathbb R)} < \epsilon$ for all $k \in \mathbb N$.
  By taking a sequence of values of $\epsilon$ tending to zero and using a diagonalization argument,
 we obtain a subsequence of $\{\theta^i_n\}_{n \in \mathbb N}$ which converges to $g_i$ strongly in $H^1(\mathbb R)$. 
 Since the same argument shows that every subsequence of $\{\theta^i_n\}_{n \in \mathbb N}$ has a
 subsubsequence which converges to $g_i$ in $H^1(\mathbb R)$,
 it follows that $\{\theta^i_n\}_{n \in \mathbb N}$ itself converges to 
$g_i$ in $H^1(\mathbb R)$. 
\end{proof}

\begin{lem} Under the same assumptions as in Lemma \ref{giconvstronginf}, we have, for each $i \in \mathbb N$, 
\begin{equation}
\lim_{n \to \infty} \intR |v_n^i|^p = \intR |g_i|^p
\label{vntogip}
\end{equation}
for all $p$ such that $2 \le p < \infty$, and 
\begin{equation}
\lim_{n \to \infty}\intR v_n^i \left({v_n^i}' \right)^2 = \intR g_i(g_i')^2.
\label{vntogiinmixedterm}
\end{equation}
 In particular, 
\begin{equation}
\begin{aligned}
\lim_{n \to \infty} E_2(v_n^i) &= E_2(g_i)=a_i\\
\lim_{n \to \infty} E_3(v_n^i) &= E_3(g_i)=b_i.
\end{aligned} 
\label{e2e3contonh1}
\end{equation}
\label{vnitogilowertermslem}
\end{lem}

\begin{proof}
Equation \eqref{vntogip} follows from Lemma \ref{giconvstronginf} and the fact that, 
by standard Sobolev embedding theorems, $L^p$ embeds continuously in $H^1(\mathbb R)$ when $2 \le p \le \infty$.
 For \eqref{vntogiinmixedterm}, we can write
\begin{equation*}
\intR v_n^i ({v_n^i}')^2 - \intR g_i(g_i')^2 = \intR (v_n^i - g_i)\left({v_n^i}' \right)^2 + \intR g_i ({v_n^i}'-g_i')({v_n^i}'+g_i').
\end{equation*}
Using H\"older's inequality and Sobolev embedding, we can majorize the integrals on the right-hand side by
\begin{equation*}
\begin{aligned}
\|v_n^i &- g_i\|_{L^\infty(\mathbb R)}  \|v_n^i\|_{H^1(\mathbb R)}^2 + 
\|g_i\|_{L^\infty(\mathbb R)}\|v_n^i - g_i\|_{H^1(\mathbb R)} \left( \|v_n^i\|_{H^1(\mathbb R)}+ \|g^i\|_{H^1(\mathbb R)}\right)\\
& \le C \|v_n^i - g_i\|_{H^1(\mathbb R)}  \left( \|v_n^i\|_{H^1(\mathbb R)}^2+ \|g^i\|_{H^1(\mathbb R)}^2\right).
\end{aligned}
\end{equation*}
Since $g_i \in H^1(\mathbb R)$, and $\{v_n^i\}_{n \in \mathbb N}$ converges to $g_i$ in $H^1(\mathbb R)$,
the preceding expression has limit zero as $n \to \infty$, proving  \eqref{vntogiinmixedterm}.  Finally, \eqref{e2e3contonh1} follows immediately from Lemma \ref{giconvstronginf} and \eqref{vntogip}.
\end{proof}

\begin{lem}   Under the same assumptions as in Lemma \ref{giconvstronginf}, we have, for each $i \in \mathbb N$, 
 \begin{equation}
 E_4(g_i) \le \liminf_{n \to \infty} E_4(\theta^i_n) = \liminf_{n \to \infty} E_4(v^i_n).
 \label{e4glte4v}
 \end{equation}
\label{e4thetanlem}
\end{lem}

\begin{proof} Choose a subsequence $\{\theta^i_{n_k}\}_{k \in \mathbb N}$ of $\{\theta^i_n\}_{n \in \mathbb N}$
 such that $\displaystyle \lim_{k \to \infty}E_4(\theta^i_{n_k})= \liminf_{n \to \infty} E_4(\theta^i_n)$. 
  By the weak compactness of bounded sets in Hilbert space, we can find a further subsequence, 
also denoted by $\{\theta^i_{n_k}\}$, which converges weakly in $H^2(\mathbb R)$ to $g_i$.   
By the lower semicontinuity of the norm in Hilbert space, we have that
$$
\|g_i\|_{H^2} \le \liminf_{k \to \infty} \|\theta^i_{n_k}\|_{H^2}
$$
 Since $\{\theta^i_{n_k}\}_{k \in \mathbb N}$ converges strongly in $H^1(\mathbb R)$ to $g_i$, this implies that
 $$
 \|g_i''\|_{L^2} \le \liminf_{k \to \infty}\|(\theta^i_{n_k})''\|_{L^2}.
 $$ 
On the other hand, by Lemma \ref{vnitogilowertermslem}, we have that
 $$
 \intR \left(-\frac56 g_i g_{ix}^2 + \frac{5}{32} g_i^4\right) =
 \lim_{k \to \infty}\intR \left(-\frac56 \theta^i_{n_k} (\theta^i_{n_k})_x^2 + \frac{5}{32} (\theta^i_{n_k})^4\right).
 $$
 Combining the last two statements, we obtain \eqref{e4glte4v}.
\end{proof}

\section{Minimizers among the set of $N$-solitons} \label{sec:finitemin}

Our analysis of minimizing sequences $\{\phi_n\}$ for the variational problem defined in \eqref{deflambda} and \eqref{defJ} will involve showing that for large $n$, the functions $\phi_n$ are well approximated by sums of widely separated solitary waves.  It will follow that the variational problem can be reduced to that in which the functions $g$ in \eqref{deflambda} are required to be soliton profiles.   For soliton profiles, the values of the functionals $E_2$, $E_3$, and $E_4$ are explicitly given by \eqref{Ekvalue}, and therefore the variational problem can be further reduced to the problem of minimizing sums of certain powers of numbers, subject to the constraint that sums of other powers be held constant.  Our main goal in this section is to solve the latter problem, which is essentially an exercise in multivariable calculus, though the details of the solution are somewhat involved.   The solution is accomplished in Lemmas \ref{ruleoutthreepos} and \ref{ruleoutthreeposinf}.  An immediate consequence of Lemma \ref{ruleoutthreepos} is that, 
among the set of all $N$-soliton profiles for the KdV equation, the ones which minimize $E_4$ subject to the
 constraints that $E_3$ and $E_2$ be held constant are precisely the $1$-soliton and $2$-soliton profiles.  

We note that, although here we only prove a variational characterization of $N$-solitons for the case $N=2$, the results of this section are sufficient to handle the case of general $N$.  As was remarked in the introduction, the obstacle to applying the method of this paper to obtain a variational characterization for general $N$-solitons is not the analysis in this section, but rather the necessity of generalizing the uniqueness result, Theorem \ref{unique}, to arbitrary values of $N$.

\begin{lem}
Suppose $A, B > 0$ and $k \in \mathbb N$.  If the system of equations
\begin{equation}
\begin{aligned}
\sum_{i=1}^k x_i^3 &= A^3\\
\sum_{i=1}^k x_i^5 &= B^5
\end{aligned}
\label{ABsyst}
\end{equation}
has a solution $(x_1,\dots,x_k)$ with $x_k \ge 0$ for $i=1,\dots,k$; then  
\begin{equation}
\left(\frac{1}{k}\right)^{2/15} \le \frac{B}{A} \le 1.
\label{ABexistcond}
\end{equation}
\label{ABexistlem}
\end{lem}

\begin{proof}
Suppose the system \eqref{ABsyst} has a solution $(x_1,\dots,x_k)$ with $x_i \ge 0$ for $i=1,\dots,k$.  Defining $p_i = x_i^3/A^3$ for each $i$, we have that  $0 \le p_i \le 1$, so $p_i^{5/3} \le p_i$.  Therefore 
$$
(B/A)^5=  \sum_{i=1}^k p_i^{5/3}\le \sum_{i=1}^k p_i =1,
  $$
 which implies that
$B/A \le 1$.   Also, by H\"older's inequality we have
$$
A^3=\sum_{i=1}^k x_i^3 \le \left(\sum_{i=1}^k x_i^5\right)^{3/5}\left(\sum_{i=1}^k 1\right)^{2/5} =B^3 k^{2/5},
$$
which implies that $(1/k)^{2/15} \le B/A$.
\end{proof}

\begin{lem} Suppose $A, B > 0$ and consider the systems
\begin{equation}
\begin{aligned}
y_1^3 + y_2^3 &= A^3\\
y_1^5 + y_2^5 &= B^5
\end{aligned}
\label{alphbetequation}
\end{equation}
and
\begin{equation}
\begin{aligned}
2 y_1^3 + y_2^3 &= A^3\\
2 y_1^5 + y_2^5 &= B^5
\end{aligned}
\label{delgamequation}
\end{equation}
for $(y_1,y_2)$ in the first quadrant $U=\left\{(y_1,y_2)\in \mathbb R^2: y_1 \ge 0, y_2 \ge 0\right\}$.
\begin{enumerate}
\item
If $B/A=1$,  then \eqref{alphbetequation} has exactly two solutions in $U$, given by $(0,A)$ and $(A,0)$, and and \eqref{delgamequation} has exactly one solution in $U$, given by $(0,A)$.

\item If $(1/2)^{2/15}<B/A <1$, then \eqref{alphbetequation} has exactly two solutions in $U$, which are of the form $(\alpha,\beta)$ and $(\beta,\alpha)$,where $0<\alpha < \beta$; and \eqref{delgamequation} has exactly one solution in $U$, which is of the form $(\gamma,\delta)$ where $0<\gamma<\delta$.

\item If $B/A=(1/2)^{2/15}$, then \eqref{alphbetequation} has exactly one solution in $U$, which is given by $(A/2^{1/3},A/2^{1/3})$; and \eqref{delgamequation} has exactly two solutions in $U$: one given by $(A/2^{1/3},0)$, and one of the form $(\gamma,\delta)$ where $0 < \gamma < \delta$.

\item If $(1/3)^{2/15} < B/A < (1/2)^{2/15}$, then \eqref{alphbetequation} has no solutions in $U$, and \eqref{delgamequation} has exactly two solutions in $U$, which are of the form $(\gamma_1,\delta_1)$ and $(\gamma_2,\delta_2)$, where $0 < \gamma_1 < \delta_1$ and $0 < \delta_2 < \gamma_2$.

\item If $B/A=(1/3)^{2/15}$, then \eqref{alphbetequation} has no solutions in $U$, and \eqref{delgamequation} has exactly one solution in $U$, which is given by $(A/3^{1/3},A/3^{1/3})$.

\end{enumerate}
\label{delgamlemma}
\end{lem}

\begin{proof} Suppose $(y_1,y_2)$ solves \eqref{alphbetequation} and $y_1, y_2>0$.  Letting $\theta=y_2/y_1$, we obtain that $y_1= A/(1+\theta^3)^{1/3}$ and  $g(\theta)=(B/A)^{15}$, where
\begin{equation}
g(t) := \frac{(1+t^5)^3}{(1+t^3)^5}.
\label{defgtheta}
\end{equation}
Conversely, for each choice of $\theta > 0$ such that $g(\theta)=(B/A)^{15}$, we have a solution $(y_1,y_2)$ of \eqref{alphbetequation} given by the positive numbers $y_1 = A/(1+\theta^3)^{1/3}$ and $y_2 = \theta y_1$.  Therefore, for a given choice of $B/A$, the number of solutions $(y_1,y_2)$ of \eqref{alphbetequation} with $y_1 > 0$ and $y_2 > 0$ is equal to the number of solutions $\theta > 0$ to the equation $g(\theta)=(B/A)^{15}$.

We have that $g(0)=1$, $g(1)=1/4$, $\displaystyle \lim_{t \to \infty} g(t)=1$, and
$$
g'(t)=\frac{15(1+t^5)^2(t^4-t^2)}{(1+t^3)^6},
$$
so $g(t)$ is monotone decreasing for $0 \le t \le 1$ and monotone increasing for $1 \le t < \infty$.
Therefore the equation $g(\theta)=(B/A)^{15}$ has exactly two solutions when $B/A \in ((1/2)^{2/15},1)$, and has exactly one solution when $B/A=(1/2)^{2/15}$.   The assertions of the lemma concerning \eqref{alphbetequation} then follow.

On the other hand, when $y_1 > 0$, we have that $(y_1,y_2)$ solves \eqref{delgamequation} if and only if $y_1=A/(2+\theta^3)^{1/3}$, $y_2 = \theta y_1$, and $h(\theta)=(B/A)^{15}$, where
$$
h(t) := \frac{(2+t^5)^3}{(2+t^3)^5}.
$$
We have $h(0)=1/4$, $h(1)=1/9$ and $\lim_{t \to \infty} h(t)=1$, and
$$
h'(t)=\frac{30(2+t^5)^2(t^4-t^2)}{(2+t^3)^6},
$$
so that $h(t)$ is monotone decreasing for $0 \le t \le 1$ and monotone increasing for $1 \le t < \infty$.  When $B/A = (1/3)^{3/15}$, the equation $h(\theta)=(B/A)^{15}$ has  exactly one solution, namely $\theta = 1$.  When $B/A \in ((1/3)^{2/15},(1/2)^{2/15})$, the equation $h(\theta)=(B/A)^{15}$ has exactly two solutions, one of which is greater than one and one of which is less than one.  When $B/A =(1/2)^{2/15}$ there are again exactly two solutions, one of which is $\theta =0$ and the other of which is a value $\theta > 1$.  Finally, when  $B/A \in ((1/2)^{2/15},1)$, there is exactly one solution to $h(\theta)=(B/A)^{15}$, and it satisfies $\theta > 1$.  These statements imply the  assertions of the lemma concerning \eqref{delgamequation}.
\end{proof}

\begin{defn}  Let
$$
T=\{(A,B) \in \mathbb R^2: A>0, B>0, (1/2)^{2/15} \le B/A \le 1\}.
$$
For each $(A,B) \in T$, we define
$$
m(A,B)=y_1^7 + y_2^7,
$$
where  $(y_1, y_2)$ is the unique solution to  \eqref{alphbetequation} satisfying $0 \le y_1 \le y_2$ and $y_2 > 0$, guaranteed by Lemma \ref{delgamlemma}.
\label{defm}  
\end{defn}

\begin{lem}\

\begin{enumerate}

\item For all $(A,B) \in T$, and for every $\lambda > 0$, we have
\begin{equation}
m(\lambda A,\lambda B) = \lambda^7 m(A,B).
\label{fhomog}
\end{equation}

\item The function $m(A,B)$ is continuous on the set $T$.
\end{enumerate}
\label{mcont}
\end{lem}

\begin{proof}
The homogeneity property \eqref{fhomog} of $m(A,B)$ is an easy consequence of the definition of $m(A,B)$.  To see that
$m(A,B)$ is continuous on $T$, observe first that for given $(A,B) \in T$, the numbers $y_1$ and $y_2$ given in Definition
\ref{defm} are given by $\displaystyle y_2 = A/(\tilde \theta^3 + 1)^{1/3}$ and $y_1 = \tilde \theta y_2$, where
$\tilde \theta=\tilde\theta(A,B)$ is the unique solution in $[0,1]$ of the equation $g(\tilde \theta)=(B/A)^{15}$, and $g$ is the
function defined in \eqref{defgtheta}.  Since $g$ is continuous and monotone decreasing on $[0,1]$, and $g([0,1])=[1/4,1]$,
then the inverse map $h:[1/4,1] \to [0,1]$ defined by $h(g(t)) = t$ is also continuous.  Therefore
$\tilde \theta(A,B) = h((B/A)^{15})$ is continuous on $T$, so $y_2$ and hence also $y_1$ depend continuously on $(A,B)$. 
So $m(A,B)=y_1^7 +y_2^7$ is continuous on $T$ as well.
 \end{proof}

\begin{lem}
Suppose $y_1 \ge y_2 \ge 0$ and $z_1 \ge z_2 \ge 0$, and
\begin{equation}
\begin{aligned}
z_1^3+z_2^3 &\le y_1^3 + y_2^3\\
z_1^5+z_2^5 &\ge y_1^5 + y_2^5.
\label{ziandyi}
\end{aligned}
\end{equation}
Then
\begin{equation}
z_1^7+z_2^7 \ge y_1^7 + y_2^7.
\label{zy7}
\end{equation}
If equality holds in \eqref{zy7}, then $y_1=z_1$, $y_2=z_2$, and equality holds in both parts of \eqref{ziandyi}.
\label{zy7lem}
\end{lem}

\begin{proof} We may assume $y_1 > 0$ and $z_1 > 0$, or otherwise there is nothing to prove.
 Let $A^3=y_1^3+y_2^3$, $B^5=y_1^5+y_2^5$, and $C^7=y_1^7+y_2^7$.  By Lemmas \ref{ABexistlem} and \ref{delgamlemma}, we have $(1/2)^{2/15} \le B/A \le 1$.  Define $\theta = y_2/y_1 \in [0,1]$ and $\tilde \theta = z_2/z_1 \in [0,1]$, and define the function $g$ as in \eqref{defgtheta}.  Then from the definitions of $A$ and $B$ we deduce that $g(\theta)=(B/A)^{15} \in [1/4,1]$; and from \eqref{ziandyi} we have that
 \begin{equation}
 \frac{B}{(1+\tilde \theta^5)^{1/5}} \le z_1 \le \frac{A}{(1+\tilde \theta^3)^{1/3}}, 
 \label{z1squeeze}
 \end{equation}
 which implies that $g(\tilde \theta) \ge (B/A)^{15}=g(\theta)$.  Since, as shown in the proof of Lemma \ref{delgamlemma}, $g$ is monotone decreasing on $[0,1]$, it follows that $\tilde \theta \le \theta$.   Now define 
 $$
 k(t) := \frac{(1+\theta^7)^5}{(1+\theta^5)^7}.
 $$
 Then as in the proof of Lemma \ref{delgamlemma}, an elementary computation (whose details we omit) shows that $k(t)$ is, like $g(t)$, strictly decreasing on $[0,1]$ and strictly increasing on $[1,\infty)$.   Therefore $k(\tilde \theta) \ge k(\theta)=(C/B)^{35}$, and so
 \begin{equation}
 \frac{C}{(1+\tilde \theta^7)^{1/7}} \le \frac{B}{(1+\tilde \theta^5)^{1/5}}.
 \label{C7B5} 
 \end{equation}
Taken with \eqref{z1squeeze}, this implies that
\begin{equation}
\frac{C}{(1+\tilde \theta^7)^{1/7}} \le z_1,
\label{C7}
\end{equation}
which yields \eqref{zy7}.

If equality holds in \eqref{zy7}, then equality also holds in \eqref{C7}, so from \eqref{z1squeeze} and \eqref{C7B5} we have that equality holds in \eqref{C7B5}.
Therefore $k(\tilde \theta)=k(\theta)$. Since $k$ is strictly decreasing on $[0,1]$, this implies that
$\tilde \theta = \theta$, and hence $g(\tilde \theta)=g(\theta)$ and so $\tilde B/\tilde A = B/A$. But from \eqref{ziandyi} we have that $\tilde A \le A$ and $\tilde B \ge B$, so it follows
that $\tilde A = A$ and $\tilde B = B$.   Hence $z_1$ and $z_2$ satisfy the same equation \eqref{alphbetequation} as $y_1$ and $y_2$, so by  Lemma \ref{delgamlemma}, we must have $z_1=y_1$ and $z_2=y_2$.
\end{proof}

\begin{lem}  Suppose $A, B > 0$ and $(1/3)^{2/15} < B/A < 1$.  For $x=(x_1,x_2,x_3) \in \mathbb R^3$, define
\begin{equation*}
\begin{aligned}
g_1(x) & = x_1^3 + x_2^3 + x_3^3\\
g_2(x) & = x_1^5 + x_2^5+ x_3^5\\
f(x) & = x_1^7 + x_2^7 + x_3^7,
\end{aligned}
\end{equation*}
and define
\begin{equation}
\begin{aligned}
\Gamma &= \left\{x \in \mathbb R^3: g_1(x)=A^3\   \text{and $g_2(x)=B^5$}\right\}\\
\Omega &= \left\{x \in \mathbb R^3: \text{$x_1 > 0$, $x_2 > 0$, and $x_3 > 0$}\right\}.
\end{aligned}
\label{defgamomega}
\end{equation}
 Then $\Gamma \cap \Omega$ is nonempty, and is a smooth one-dimensional submanifold of $\mathbb R^3$.  
 
   If we assume further that $B/A \ge (1/2)^{2/15}$, then $\Gamma \cap \Omega$ must consist of three nonempty connected components $\Gamma_1$, $\Gamma_2$, and $\Gamma_3$.  For each $i=1,2,3$, let $\overline \Gamma_i$ denote the closure of $\Gamma_i$, and $\partial \Gamma_i$ the boundary of $\Gamma_i$, in the topology of $\mathbb R^3$.   Then the restriction of $f$ to $\overline \Gamma_i$ takes its maximum value at a single point in $\Gamma_i$, and takes its minimum value $f_{\text{min}}$ on  $\partial \Gamma_i$.   For each $(x_1,x_2,x_3) \in \Gamma \cap \Omega$, we have $f(x_1,x_2,x_3) > f_{\text{min}}$.
 \label{3ge2}
\end{lem}

\begin{proof} Since $(1/3)^{2/15} < B/A < 1$, then by Lemma \ref{delgamlemma}, there exists a solution $(y_1,y_2)=(\gamma,\delta)$ to \eqref{delgamequation} with $0 < \gamma < \delta$.   Setting $x_1=x_2=\gamma$ and $x_3=\delta$ then defines a point $Q = (\gamma,\gamma,\delta) \in \Gamma \cap \Omega$, and shows that $\Gamma \cap \Omega$ is nonempty.  The fact that $\Gamma \cap \Omega$ is a smooth one-dimensional submanifold of $\mathbb R^3$ follows from the implicit function theorem (see, for example, Theorem 1.38 of \cite{W})  and the fact that the gradients $\nabla g_1(x)$ and $\nabla g_2(x)$ are linearly independent at all $x \in \Gamma \cap \Omega$.  Indeed, if for some $c_1, c_2 \in \mathbb R$, with $c_1, c_2$ not both zero, we have $c_1 \nabla g_1(x)+c_2 \nabla g_2(x)=0$, then it follows easily that $x_1= x_2 = x_3$.  But then $g_1(x)=A^3$ and $g_2(x)=B^5$ imply that $B/A = (1/3)^{2/15}$, contradicting our assumption about $B/A$.
 
  Now suppose we are at a point $x_0=(x_{10},x_{20}, x_{30}) \in \Gamma \cap \Omega$ where $x_{20} \ne x_{30}$.  (Note that the point $Q$ defined above is such a point.) Then from the implicit function theorem it follows that there exists a neighborhood $I$ of $x_0$ in $\mathbb R$ such that for all $t \in I$, there are unique numbers $x_2(t)$ and $x_3(t)$ so that $(t,x_2(t),x_3(t)) \in \Gamma \cap \Omega$.  Moreover, $x_2(t)$ and $x_3(t)$ are smooth functions of $t \in I$, with
\begin{equation}
\begin{aligned}
\frac{dx_3}{dt}&= \frac{t^2(x_2^2-t^2)}{x_3^2(x_3^2-x_2^2)}\\
\frac{dx_2}{dt}&=\frac{t^2(t^2-x_3^2)}{x_2^2(x_3^2-x_2^2)}
 \end{aligned}
\label{x2x3deriv}
\end{equation}
on $I$.
 
In particular, this analysis when applied to the point $Q$ shows that that there are functions $x_2(t)$ and $x_3(t)$ defined for $t$ in a neighborhood $I$ of $\gamma$ such that $(t,x_2(t),x_3(t)) \in \Gamma \cap \Omega$ for all $t \in I$, and equations \eqref{x2x3deriv} hold on $I$. From \eqref{x2x3deriv} we have that  $\frac{dx_2}{dt}<0$ at $t=\gamma$, so there exists an $\epsilon > 0$ such that $0 < x_2(t) < t < x_3(t)$ for all $t$ such that $\gamma < t < \gamma + \epsilon$.

Assume now further that $B/A \ge (1/2)^{2/15}$.  Then by Lemma \ref{delgamlemma}, the point $(\gamma,\delta)$ defined above is the only solution $(y_1,y_2)$ to \eqref{delgamequation} with $y_1>0$ and $y_2 > 0$.  Let $S$ be the set of all $t_0 > \gamma$ such that there exist smooth functions $x_2(t)$, $x_3(t)$ defined for all $t \in (\gamma,t_0)$ such that $(t,x_2(t),x_3(t)) \in \Gamma \cap \Omega$ and 
$$
 0 < x_2(t) < t < x_3(t) 
$$ 
for all $t \in (\gamma,t_0)$.  Then $S$ is nonempty and bounded, since $\epsilon \in S$ and $t_0 \le A$ for all $t_0 \in S$.   Therefore $S$ has a finite supremum, which we denote by $t_m$.  Equations \eqref{x2x3deriv} imply that $\frac{dx_3}{dt}\le 0$ and $\frac{dx_2}{dt} \le 0$ for all $t \in [\gamma,t_m)$, so $x_{2}(t)$ and $x_3(t)$ have limits as $t$ approaches $t_m$ from the left; we denote these limits by $x_2(t_m)$ and $x_3(t_m)$ respectively. 

We have that $0 \le x_2(t_m) \le t_m \le x_3(t_m)$.  It cannot be the case that $0 < x_2(t_m) < t_m < x_3(t_m)$, for then an application of the implicit function theorem would allow us to extend $x_2(t)$ and $x_3(t)$ to an open interval containing $t=t_m$, contradicting the maximality of $t_m$.  Since $x_2(t)$ is nonincreasing on $[\gamma,t_m)$ and $x_2(\gamma)=\gamma$,  we have $x_2(t_m) < t_m$. Also, we cannot have  $0 < x_2(t_m) < t_m = x_3(t_m)$,  for then setting $\tilde \gamma =t_m= x_3(t_m)$ and $\tilde \delta =x_2(t_m)$ would produce a solution $(y_1,y_2)=(\tilde \gamma, \tilde \delta)$ of \eqref{delgamequation} with $0 < \tilde \delta < \tilde \gamma$, which is distinct from $(\gamma,\delta)$ and therefore contradicts the uniqueness of positive solutions to \eqref{delgamequation}.   Finally, if $0 < x_2(t_m)=t_m=x_3(t_m)$, we would obtain a point in $\Gamma \cap \Omega$ where $x_1=x_2=x_3$, which we have already seen is impossible.    We have thus ruled out all the possibilities in which $x_2(t_m)>0$, so it follows that $x_2(t_m)=0$. 

 We have shown that $\Gamma \cap \Omega$ contains the smooth arc
$$
\left\{(t, x_2(t), x_3(t)): \gamma \le t \le t_m\right\}
$$
whose endpoints are $Q$ and $P_1=(t_m,0,x_3(t_m))\in \partial\Omega$.  By symmetry, $\Gamma \cap \Omega$ also contains a smooth arc whose endpoints are $Q$ and $P_2=(0,t_m,x_3(t_m)) \in \partial\Omega$.  The interior of the union of these two arcs is a connected component $\Gamma_1$ of $\Gamma \cap \Omega$.

 We now consider the problem of maximizing or minimizing $f(x)$ subject to the constraint $x \in \overline \Gamma_1$.  If the maximum or minimum occurs at an interior point $x=(x_1,x_2,x_3) \in \Gamma_1$, then $x$ must be a critical point of the constrained variational problem, in the sense that 
 $$
 \nabla f(x)= \lambda_1 \nabla g_1(x) + \lambda_2 \nabla g_2(x)
 $$
 for some $\lambda_1, \lambda_2 \in \mathbb R$.  This implies that
 $$
 7x_i^6=3\lambda_1 x_i^2 + 5 \lambda_2 x_i^4
 $$
 for $i=1,2,3$.  Letting $z_i = x_i^2$, we have for all $i,j \in \left\{1,2,3\right\}$ that
 $$
 \frac{5\lambda_2}{7}(z_i-z_j)=z_i^2 - z_j^2.
 $$
 Therefore either $z_i = z_j$ or $z_i+z+j=5\lambda_2/7$.   It follows that the set $\left\{z_1,z_2,z_3\right\}$ cannot consist of three distinct numbers: if, for example, $z_1 \ne z_3$ and $z_2 \ne z_3$, then we must have that
 $z_1+z_3=z_2+z_3 = 5 \lambda_2/7$, so $z_1=z_2$.  It follows then from Lemma \ref{delgamlemma} that the only possible critical points of $f$ on $\Gamma \cap \Omega$ are $Q=(\gamma, \gamma,\delta)$,
 $(\gamma, \delta,\gamma)$, and $(\gamma, \delta,\gamma)$.  But since $\delta > \gamma$, and $x_3 > x_2$ at all points on $\Gamma_1$ except $Q$, we conclude that $Q$ is the only critical point of $f$ on $\Gamma_1$.

 We have now shown that either $f$ takes its maximum value over $\overline \Gamma_1$ at $Q$ and its minimum value at $P_1$ and $P_2$, or $f$ takes its minimum value over $\overline \Gamma_1$ at $Q$ and its maximum value at $P_1$ and $P_2$.   To decide between these two alternatives, it suffices to determine whether the restriction of $f$ to $\Gamma_1$ has a local maximum or a local minimum at $Q$.  For this purpose, we use the second derivative test for constrained extrema, as expounded for example in \cite{Mu}.

 Consider the Lagrangian $L(x)$ defined by
 $$
 L(x)=f(x)-\lambda_1(g_1(x)-A^3)-\lambda_2(g_2(x)-B^5),
 $$
 and form
 the ``augmented Hessian'', a $5 \times 5$ matrix $\mathbf H$ defined by
 $$
\mathbf H=\left[
 \begin{matrix}
 \mathbf 0 & \mathbf B\\
 \mathbf C & \mathbf D
 \end{matrix}
 \right],
 $$
 where $\mathbf B$ is the $2 \times 3$ matrix given by
 $$
 \mathbf B = \left[
 \begin{matrix}
  -(g_1)_{x_1} & -(g_1)_{x_2} & -(g_1)_{x_3} \\
  -(g_2)_{x_1} & -(g_2)_{x_2} & -(g_2)_{x_3}
 \end{matrix}
 \right],
 $$
 $\mathbf C=\mathbf B^T$ is the transpose of $\mathbf B$, and $\mathbf D$ is the $3 \times 3$ Hessian of $L$, given by
 $$
\mathbf D_{ij}=L_{x_i x_j}\quad \text{for $i,j \in \left\{1,2,3\right\}$}.
 $$
 Here and in what follows we use $\mathbf 0$ to denote matrices of various sizes (in this case, a $2 \times 2$ matrix) with all zero entries. 

 We want to compute the determinant $\det \mathbf H$ of $\mathbf H$ at $x=Q$. Calculations show that at $x=Q$, we have
 $$
 \mathbf B=\left[
 \begin{matrix}
  -3\gamma^2 & -3\gamma^2 & -3\delta^2 \\
    -5\gamma^4 & -5\gamma^4 & -5\delta^4
 \end{matrix}
 \right]
 $$
 and
 $$\mathbf D=\left[
 \begin{matrix}
  -14\gamma^3(\gamma^2+\delta^2) & 0 & 0 \\
    0 &-14\gamma^3(\gamma^2+\delta^2) & 0\\
    0 & 0 & -14\delta^3(\gamma^2+\delta^2)\\
 \end{matrix}
 \right];
 $$
from which one finds that
$$\det (\mathbf 0-\mathbf B \mathbf D^{-1}\mathbf C) = \frac{-450\gamma\delta}{196}$$
and
$$\det \mathbf D=-14^3(\gamma^2-\delta^2)^3 \delta^3\gamma^6.
$$
Let $\mathbf I_2$ be the $2\times 2$ identity matrix, and $\mathbf I_3$ the $3 \times 3$ identity matrix.   Then the matrix
$$
\left[\begin{matrix}\mathbf I_2 & \mathbf 0\\-\mathbf D^{-1}\mathbf C & \mathbf I_3\end{matrix}\right]
$$
has determinant equal to one, and so we can write
\begin{equation*}
\begin{aligned}
\det \mathbf H &= \det\left[\begin{matrix}
 \mathbf 0 &\mathbf B\\
 \mathbf C & \mathbf D\end{matrix}\right]\left[\begin{matrix}\mathbf I_2 & \mathbf 0\\-\mathbf D^{-1}\mathbf C & \mathbf I_3\end{matrix}\right]
 =\det \left[\begin{matrix}\mathbf A-\mathbf B\mathbf D^{-1}\mathbf C & \mathbf B\\ \mathbf 0 & \mathbf D\end{matrix}\right]=\\
 &=\det(\mathbf A-\mathbf B\mathbf D^{-1}\mathbf C)\det\mathbf  D
 =14\cdot 450 \gamma^7\delta^4(\gamma^2-\delta^2)^3.
 \end{aligned}
 \end{equation*}

 Since $\gamma < \delta$, we have shown that $\det \mathbf H < 0$ at $x=Q$.  It is easy to check that $\mathbf B$ has full rank at $x=Q$.  Therefore, according to Theorem 36 on p.\ 58 of \cite{Mu}, we have that $\mathbf v^T \mathbf D \mathbf v < 0$ for all nonzero column vectors $\mathbf v \in \mathbb R^3$ satisfying $\mathbf B\mathbf v = \mathbf 0$.  In other words, the Hessian $\mathbf D$ of $L$ is negative definite in all directions $\mathbf v$ which are tangent to both the surfaces $\{x: g_1(x)=A^3\}$ and $\{x: g_2(x)=B^5\}$ at $Q$.  From a classical result in the calculus of variations (see for example page 334 of \cite{LY}), it follows that $f(x)$ has a local maximum at $Q$ subject to the restriction $x \in \Gamma_1$.

We have now proved that the restriction of $f$ takes its maximum over $\overline \Gamma_1$ at $Q=(\gamma,\gamma,\delta) \in \Gamma_1$ and its minimum value at the endpoints $P_1=(t_m,0,x_3(t_m))$ and $P_2=(0,t_m,x_3(t_m))$ of $\overline \Gamma_1$.  Let us define $f_{\text{max}} = f(Q)$ and $f_{\text{min}}=f(P_1)=f(P_2)$.   Since the restriction of $f$ to $\Gamma_1$ has no critical points in
$\Gamma_1 \backslash Q$, we must have $f(x) > f_{\text{min}}$ for all  $x \in \Gamma_1$.

By symmetry, it follows that $\Gamma \cap \Omega$ also contains a component $\Gamma_2$ which includes the point $(\gamma,\delta,\gamma)$ and whose closure has endpoints $(0,x_3(t_m),t_m)$ and $(t_m,x_3(t_m),0)$; and a component $\Gamma_3$ which includes the point $(\delta,\gamma,\gamma)$ and whose closure has endpoints $(x_3(t_m),0,t_m)$ and $(x_3(t_m), t_m,0)$.  Furthermore, we know that the maximum value of $f$ on $\overline \Gamma_2$ is attained at $(\gamma,\delta,\gamma)$, and is equal to $f_{\text{max}}$; the minimum value of $f$ on $\overline \Gamma_2$ is attained at the  boundary points of $\overline \Gamma_2$, and is equal to $f_{\text{min}}$; and $f(x) > f_{\text{min}}$ for all $x \in \Gamma_2$.  Similar statements hold for $\overline \Gamma_3$.

To complete the proof of the Lemma, it remains only to show that $\Gamma \cap \Omega$ contains no other components besides $\Gamma_1$, $\Gamma_2$, and $\Gamma_3$.   To prove this, assume $Q_0=(x_{10},x_{20},x_{30}) \in \Gamma \cap \Omega$; we wish to show that $Q_0 \in \Gamma_i$ for some $i \in \left\{1,2,3\right\}$.  We know $x_{10}$, $x_{20}$, and $x_{30}$ cannot all be equal (for this would imply $B/A=(1/3)^{2/15}$); and if any two of $x_{10}$, $x_{20}$, $x_{30}$ are equal, then by Lemma \ref{delgamlemma}, $Q_0$ must be one of the points $(\gamma,\gamma,\delta)$, $(\gamma,\delta,\gamma)$, or $(\delta,\gamma,\gamma)$, and therefore lies in one of the $\Gamma_i$.  We may therefore  assume without loss of generality that $x_{10}<x_{20} < x_{30}$.  
Then the analysis above shows that there exists some $\epsilon > 0$ and a smooth curve $x(t)=(t, x_2(t),x_3(t))$ mapping $I=(x_{10}-\epsilon,x_{10}+\epsilon)$ into $\Gamma \cap \Omega$, such that $x(0)=Q_0$,  and satisfying  
$t< x_2(t) < x_3(t)$ on $I$.  Moreover, $x(t)$ satisfies equations \eqref{x2x3deriv}, which imply that $\frac{dx_2}{dt} <0$ and $\frac{dx_3}{dt} > 0$ on $I$.  

Now  let $S$ be the set of all $t_0 > x_{10}$ such that there exist smooth functions $x_2(t)$, $x_3(t)$ defined for all $t \in (x_{10},t_0)$ such that $(t,x_2(t),x_3(t)) \in \Gamma \cap \Omega$ and 
$$
 0 < t < x_2(t) < x_3(t) 
$$ 
for all $t \in (x_{10},t_0)$.   Again we define $t_m = \sup S$ and let $x_2(t_m)$ denote the limit of $x_2(t)$ as $t$ approaches $t_m$ from the left.   The implicit function theorem and the maximality of $t_m$ imply that we must have $x_2(t_m)=t_m$.   But this then implies that the point $(t_m,x_2(t_m),x_3(t_m))=(\gamma,\gamma,\delta) \in \Gamma_1$.    Therefore the set $S_1 = \left\{t \in [x_{01},t_m]:  (t,x_2(t), x_3(t)) \in \Gamma_1\right\}$ is non-empty.  The uniqueness assertion in the implicit function theorem tells us that for every $t \in [x_{01}, t_m]$, the equations $g_1(x)=A^3$ and $g_2(x)=B^5$ determine $x_2$ and $x_3$ uniquely as functions of $x_1$ in some open neighborhood of $(t,x_2(t),x_3(t))$.  Therefore $S_1$ is open.  On the other hand,  $S_1$ is clearly closed, by the continuity of $x_2(t)$ and $x_3(t)$ and the fact that $\Gamma_1$ is a closed subset of $\mathbb R^3$.  So we must have  $S=[x_{01},t_m]$, and therefore $Q_0 \in \Gamma_1$. \end{proof}

{\it Remark:}  In the case $(1/3)^{2/15}<B/A<(1/2)^{2/15}$, a similar analysis shows that $\Gamma \cap \Omega$ is homeomorphic to a circle, and contains all six of the points $P_1(\gamma_1,\gamma_1,\delta_1)$, $P_2(\gamma_1,\delta_1,\gamma_1)$, $P_3(\delta_1,\gamma_1,\delta_1)$, $P_4(\gamma_2,\gamma_2,\delta_2)$, $P_5(\gamma_2,\delta_2,\gamma_2)$, and $P_6(\delta_2,\gamma_2,\delta_2)$, where $(\gamma_1,\delta_1)$ and $(\gamma_2,\delta_2)$ are as described in part 4 of Lemma \ref{delgamlemma}.  Moreover, points $P_1$, $P_2$, and $P_3$ are local maxima for the restriction of $f$ to $\Gamma \cap \Omega$; while points $P_4$, $P_5$, and $P_6$ are local minima.   However, we will not need these facts in what follows.

\bigskip

\begin{lem}
Suppose $x_1,\dots, x_n$ are numbers such that $x_1 \ge x_2 \ge \dots \ge x_n \ge 0$, with $x_1>0$, and for each $m \in \{1,\dots, n\}$ define
\begin{equation*}
\begin{aligned}
A_m &=\left( \sum_{i=1}^m x_i^3\right)^{1/3}\\
B_m &= \left( \sum_{i=1}^m x_i^5\right)^{1/5}.
\end{aligned}
\end{equation*}
Then for each $m \in \{2,\dots,n\}$,  
\begin{equation}
\frac{B_{m-1}}{A_{m-1}} \ge \frac{B_m}{A_m},
\label{BoverAdecreases}
\end{equation}
and the inequality is strict if $x_m > 0$.
\label{BoverAlem}
\end{lem}

\begin{proof} The statement is obvious if $x_m =0$, so we may assume $x_m > 0$. Let $\displaystyle f(x)= \frac{(B_m^5-x^5)^3}{(A_m^3-x^3)^5}$.  Then $f(x_m)=(B_{m-1}/A_{m-1})^{15}$ and $f(0)=(B_m/A_m)^{15}$, so it suffices to show that $f(x_m) > f(0)$.   Now
$$
f'(x) = \frac{15  x^2(B_m^5-A_m^3x^2)(B_m^5-x^5)^2}{(A_m^3-x^3)^6}.
$$
So $f'(x) > 0$ for $0 \le x < x_0$, where $x_0=\sqrt{B_m^5/A_m^3}$.  But since $x_m \le x_i$ for all $i \in \{1,\dots,m\}$, we have
$$
A_m^3 x_m^2 = \left(\sum_{i=1}^m x_i^3\right)x_m^2 \le \sum_{i=1}^m x_i^5 = B_m^5,
$$
so $x_m \le x_0$.  Therefore $f(x_n) > f(0)$, as desired.
\end{proof}

\begin{lem}
Let $A, B>0$ be such that $(1/2)^{2/15} \le B/A \le 1$, and let $n \in \mathbb N$, with $n \ge 3$.  Suppose $x_1,\dots,x_n$ are numbers such that
$x_1 \ge \dots \ge x_n \ge 0$, with $x_3 > 0$, and
\begin{equation}
\begin{aligned}
\sum_{i=1}^n x_i^3 &= A^3\\
\sum_{i=1}^n x_i^5 &= B^5.
\end{aligned}
\label{xisyst}
\end{equation}
Then
\begin{equation}
\sum_{i=1}^n x_i^7 \ge m(A,B) + E,
\label{xiineq}
\end{equation}
where $E=E(x_1,x_2,x_3)$ is defined by
\begin{equation}
E(x_1,x_2,x_3) :=x_1^7 + x_2^7+x_3^7 - m((x_1^3 + x_2^3 + x_3^3)^{1/3},(x_1^5 + x_2^5+x_3^5)^{1/5}).
\label{defE}
\end{equation}
In particular, 
\begin{equation}
E(x_1,x_2,x_3)>0.
\label{Epos}
\end{equation}
\label{strictxn7}
\end{lem}

\begin{proof}  Let $\tilde A = (x_1^3 + x_2^3 + x_3^3)^{1/3}$ and $\tilde B = (x_1^5 + x_2^5+x_3^5)^{1/5}$. If we define $\Gamma$ and $\Omega$ as in \eqref{defgamomega} with $A$ replaced by $\tilde A$ and $B$ replaced by $\tilde B$, then since $x_1 \ge x_2 \ge x_3 > 0$, the point $x=(x_1,x_2,x_3)$ lies in $\Gamma \cap \Omega$.   The inequality \eqref{Epos} thus follows from Lemma \ref{3ge2}.

To prove \eqref{xiineq}, we use induction on $n$.  When $n=3$, the result is trivial.   Suppose $n \ge 4$ and assume the statement of the lemma is true for $n-1$; we wish
to prove it for $n$.

Suppose that $x_1 \ge \dots \ge x_n \ge 0$, with $x_3>0$, and that \eqref{xisyst} holds.  If $x_n =0$, then we are done by the inductive hypothesis, so we may assume $x_n > 0$.  Let 
\begin{equation*}
\begin{aligned}
A_{n-1} &=\frac{(A^3-x_n^3)^{1/3}}{x_n}\\
B_{n-1} &=\frac{(B^5-x_n^5)^{1/5}}{x_n},\\
\end{aligned}
\end{equation*}
and define $y_i = x_i/x_n$ for $1 \le i \le n-1$.  Then $y_1 \ge \dots \ge y_{n-1}$, and
\begin{equation*}
\begin{aligned}
\sum_{i=1}^{n-1} y_i^3 &= A_{n-1}^3\\
\sum_{i=1}^{n-1} y_i^5 &= B_{n-1}^5.
\end{aligned}
\end{equation*}
From Lemma \ref{BoverAlem} it follows that $B_{n-1}/A_{n-1} > B/A$, and from Lemma \ref{ABexistlem} we have that $B_{n-1}/A_{n-1} \le 1$.    Hence $(1/2)^{2/15} \le B_{n-1}/A_{n-1} \le 1$, and we may therefore apply the inductive hypothesis to the numbers $y_1 \ge y_2 \ge \dots \ge y_{n-1} \ge 0$.  There results the inequality 
\begin{equation}
\sum_{i=1}^{n-1} y_i^7 \ge m(A_{n-1},B_{n-1}) + E_1,
\label{sumyi7}
\end{equation}
where
\begin{equation*}
E_1=y_1^7 + y_2^7+y_3^7 - m((y_1^3 + y_2^3 + y_3^3)^{1/3},(y_1^5 + y_2^5+y_3^5)^{1/5}).
\end{equation*}
From \eqref{fhomog}, however, it follows that $E_1 = E/x_n^7$, so multiplying \eqref{sumyi7} by $x_n^7$, we conclude that
\begin{equation}
\sum_{i=1}^{n-1}x_i^7 \ge x_n^7 m(A_{n-1},B_{n-1}) + E.
\label{nm1est}
\end{equation} 

 From Lemma \ref{delgamlemma} we have that there exist $w_1$, $w_2$ with $0 \le w_1 < w_2$ such that
\begin{equation*}
\begin{aligned}
w_1^3 + w_2^3 &= A_{n-1}^3\\
w_1^5 + w_2^5 &= B_{n-1}^5.
\end{aligned}
\end{equation*}
By definition of the function $m$, we have
\begin{equation}
w_1^7+w_2^7 = m(A_{n-1},B_{n-1}).
\label{fw}
\end{equation}
Letting $z_1 = x_n w_1$, $z_2 = x_n w_2$, and $z_3 = x_n$, we see that
\begin{equation*}
\begin{aligned}
z_1^3 + z_2^3 + z_3^3 &= A^3\\
z_1^5 + z_2^5 + z_3^5 &= B^5.
\end{aligned}
\end{equation*}
Therefore $(z_1,z_2,z_3)$ is in the closure of the set $\Gamma \cap \Omega$ defined in Lemma \ref{3ge2}.  From Lemma \ref{delgamlemma} we see that the boundary of $\Gamma \cap \Omega$ consists
exactly of the six points $(0,\alpha,\beta)$, $(0,\beta,\alpha)$, $(\alpha,0,\beta)$, $(\beta,0,\alpha)$, $(\alpha,\beta,0)$, and $(\beta,\alpha,0)$.  At each of these boundary points, the function $f$ defined
in Lemma \ref{3ge2} takes the same value $\alpha^7+\beta^7$, which by definition is equal to $m(A,B)$.  Hence, by Lemma \ref{3ge2}, we have that
$$
f(z_1,z_2,z_3) \ge m(A,B).
$$
Since $f(z_1,z_2,z_3)=x_n^7(1+w_1^7+w_2^7)$, using \eqref{fw} we deduce that
$$
x_n^7 + x_n^7 m(A_{n-1},B_{n-1}) \ge m(A,B).
$$
Therefore, by \eqref{nm1est},
$$
\sum_{i=1}^n x_i^7 =x_n^7 + \sum_{i=1}^{n-1}x_i^7  \ge x_n^7 + x_n^7 m(A_{n-1},B_{n-1}) + E \ge m(A,B) + E,
$$
as was desired.
\end{proof}

\begin{lem}  Suppose $y_1 \ge y_2 \ge 0$ and $y_1 > 0$.  Let $n \in \mathbb N$, and suppose  $x_1,\dots,x_n$ are numbers such that
$x_1 \ge \dots \ge x_n \ge 0$, and
\begin{equation}
\begin{aligned}
\sum_{i=1}^n x_i^3 &\le y_1^3+y_2^3\\
\sum_{i=1}^n x_i^5 &\ge y_1^5+y_2^5.
\end{aligned}
\label{xisystineq}
\end{equation}

\begin{enumerate}

\item

If $n \ge 2$ and $y_2=0$, then $x_2 =0$ and $x_1 = y_1$.

\item

If $n \ge 3$ and  $x_3 > 0$, then 
\begin{equation}
\sum_{i=1}^n x_i^7 \ge y_1^7 + y_2^7 +E(x_1,x_2,x_3),
\label{xiyi7}
\end{equation}   
where $E(x_1,x_2,x_3)>0$ is as in \eqref{defE} and \eqref{Epos}.
\end{enumerate}

\label{ruleoutthreepos}
 \end{lem}

\begin{proof} Define $A=\left( y_1^3+y_2^3\right)^{1/3}$, $B=\left(y_1^5+y_2^5\right)^{1/5}$, 
$A_n =\left( \sum_{i=1}^n x_i^3\right)^{1/3}$, and $B_n= \left(\sum_{i=1}^n x_i^5\right)^{1/5}$. 
 From Lemma \ref{ABexistlem}, Lemma \ref{BoverAlem}, and \eqref{xisystineq}, we have that
\begin{equation}
 1 = \frac{B_1}{A_1} \ge \frac{B_2}{A_2} \ge \frac{B_3}{A_3} \ge \dots \ge \frac{B_n}{A_n} \ge \frac{B}{A} \ge \left(\frac12\right)^{2/15}.
\label{ladder}
\end{equation}

To prove part 1 of the Lemma, we simply observe that if $y_2=0$ then $B/A=1$, and if $x_2 > 0$ then $B_1/A_1 > B_2/A_2$
by Lemma \ref{BoverAlem}.  Thus \eqref{ladder} immediately gives a contradiction.   So if $y_2=0$, we must have $x_2=0$ and hence $x_1 = y_1$.

To prove part 2 of the Lemma, we first observe  that    from  \eqref{ladder}, Lemma \ref{delgamlemma},
 and the definition of the function $m$, we obtain that there exist $z_1$ and $z_2$ with $0 \le z_1 \le z_2$ and $z_2 > 0$ such that 
\begin{equation}
\begin{aligned}
z_1^3 + z_2^3 &= A_n^3\\
z_1^5 + z_2^5 & = B_n^5
\end{aligned}
\end{equation}
and $z_1^7 + z_2^7 = m(A_n,B_n)$.

 Now if $n \ge 3$ and $x_3 > 0$, then from Lemma \ref{strictxn7} it follows that 
\begin{equation}
\sum_{i=1}^n x_i^7 \ge m(A_n,B_n) + E(x_1,x_2,x_3)= z_1^7 + z_2^7 + E(x_1,x_2,x_3).
\label{xizi7}
\end{equation}
But since
$$
\begin{aligned}
z_1^3 + z_2^3 &\le  A^3\\
z_1^5 + z_2^5 &\ge  B^5,
\end{aligned}
$$
it follows from Lemma \ref{zy7lem} that $z_1^7 + z_2^7 \ge y_1^7 + y_2^7$.  This, combined with \eqref{xizi7}, gives \eqref{xiyi7}. 
\end{proof}

 The preceding results of this section have all built towards the following lemma, which is the only result from this section that will be used in the sequel.  

\begin{lem}
 Suppose $y_1 \ge y_2 \ge 0$ and $y_1 > 0$.  Let   $\{x_n\}_{n \in \mathbb N}$ be a sequence such that
$x_1 \ge x_2 \ge x_3 \ge \dots  \ge 0$, and
\begin{equation}
\begin{aligned}
\sum_{i=1}^\infty x_i^3 &\le y_1^3+y_2^3\\
\sum_{i=1}^\infty x_i^5 &\ge y_1^5+y_2^5.
\end{aligned}
\label{xisystineqinf}
\end{equation}

\begin{enumerate}

\item If $y_2=0$, then $x_2 =0$, $x_1 = y_1$, and equality holds in both parts of \eqref{xisystineqinf}.  

\item
If $x_3 > 0$, then 
\begin{equation}
\sum_{i=1}^\infty x_i^7 \ge y_1^7 + y_2^7 +E(x_1,x_2,x_3),
\label{xiyi7inf}
\end{equation}   
where $E(x_1,x_2,x_3)>0$ is as in \eqref{defE} and \eqref{Epos}.   

\item If  
\begin{equation}
\sum_{i=1}^\infty x_i^7 \le y_1^7 + y_2^7
\label{xiyi7inf2}
\end{equation}   
then $x_1=y_1$, $x_2=y_2$, $x_i = 0$ for all $i \ge 3$, and equality holds in both parts of \eqref{xisystineqinf}.

\end{enumerate}

\label{ruleoutthreeposinf}
\end{lem}

 \begin{proof}   
 Define $A=\left( y_1^3+y_2^3\right)^{1/3}$, $B=\left(y_1^5+y_2^5\right)^{1/5}$, 
 $A_\infty =\left(\sum_{i=1}^\infty x_i^3\right)^{1/3}$, and $B_\infty = \left(\sum_{i=1}^\infty x_i^5\right)^{1/5}$; 
and for $n \in \mathbb N$ define $A_n =\left( \sum_{i=1}^n x_i^3\right)^{1/3}$ and 
$B_n= \left(\sum_{i=1}^n x_i^5\right)^{1/5}$. Thus $\displaystyle \lim_{n \to \infty} A_n = A$ and 
$\displaystyle \lim_{n \to \infty} B_n = B$.  From Lemmas \ref {ABexistlem} and \ref{BoverAlem} and our assumptions, we have that 
$$
1 \ge \frac{B_1}{A_1} \ge \frac{B_2}{A_2} \ge \dots \ge \frac{B_n}{A_n} \ge \dots \ge \frac{B_\infty}{A_\infty}
 \ge \frac{B}{A} \ge \left(\frac12\right)^{2/15}.
$$ 

Part 1 of the Lemma is now proved by the same argument as part 1 of Lemma \ref{ruleoutthreepos}.

To prove part 2 of the Lemma, we suppose $n \ge 3$ and $x_3 > 0$, and consider first the case when the first inequality
 in \eqref{xisystineqinf} is strict:  that is, when $\sum_{i=1}^\infty x_i^3 < y_1^3+y_2^3$.   For each $i \in \mathbb N$
 and $n \in \mathbb N$, define $\alpha_n = B_n/B_\infty$ and $x_{in} = x_i/\alpha_n$.   
Then $\lim_{n \to \infty} \alpha_n = 1$, and for each $n \in \mathbb N$ we have
$$
\sum_{i=1}^n x_{in}^5 = \sum_{i=1}^\infty x_i^5  \ge y_1^5 + y_2^5.
$$
Also,
$$
\lim_{n \to \infty} \sum_{i=1}^n x_{in}^3 = \sum_{i=1}^\infty x_i^3 < y_1^3 + y_2^3,
$$
so by choosing $n$ sufficiently large we have $\sum_{i=1}^n x_{in}^3  \le y_1^3 + y_2^3$.
We can thus apply Lemma \ref{ruleoutthreepos} to $x_{1n} \ge \dots \ge x_{nn} \ge 0$ for all
 sufficiently large $n \in \mathbb N$, and obtain that
$$
\sum_{i=1}^n x_{in}^7 \ge y_1^7 + y_2^7 +  E(x_{1n},x_{2n},x_{3n}),
$$
or
$$
\frac{1}{\alpha_n^7}\sum_{i=1}^n x_i^7 \ge y_1^7 + y_2^7 + \frac{1}{\alpha_n^7} E(x_1,x_2,x_3).
$$
Taking the limit as $n \to \infty$ then gives us that 
$$
\sum_{i=1}^\infty x_i^7 \ge y_1^7 + y_2^7 + E(x_1,x_2,x_3).
$$
 the desired result \eqref{xiyi7inf}.

Next, consider the case when $\sum_{i=1}^\infty x_i^3 > y_1^3+y_2^3$.  Then for sufficiently large $n$,
 \eqref{xisystineq} holds, so by Lemma \ref{ruleoutthreepos} we conclude that \eqref{xiyi7} holds,
 which immediately implies \eqref{xiyi7inf}.

It remains then only to consider the case when $A_\infty = A$ and $B_\infty = B$.  In this case we argue as follows. 
For each $n \in \mathbb N$, since $B_n/A_n \ge (1/2)^{2/15}$, by Lemma \ref{delgamlemma}
 we can choose $z_{1n} \ge z_{2n} \ge 0$ such that $A_n = z_{1n}^3 + z_{2n}^3$ and 
$B_n = z_{1n}^5 + z_{2n}^5$; we then have that $m(A_n,B_n)=z_{1n}^7 + z_{2n}^7$. 
  From Lemma \ref{ruleoutthreepos}, we have that 
\begin{equation}
\sum_{i=1}^n x_i^7 \ge m(A_n,B_n) + E(x_1,x_2,x_3).
\label{trunc}
\end{equation}
But, by Lemma \ref{mcont}, we have $\displaystyle \lim_{n \to \infty} m(A_n,B_n) = m(A_\infty,B_\infty)=m(A,B)$.
Taking the limit on both sides of \eqref{trunc} as $n \to \infty$ then gives \eqref{xiyi7inf}.  This completes the proof of part 2.

To prove part 3, note that if \eqref{xiyi7inf2} holds, then by part 2 we must have $x_3 = 0$.   Then the inequalities \eqref{xisystineqinf} become
\begin{equation*}
\begin{aligned}
x_1^3 + x_2^3 &\le y_1^3 + y_2^3\\
x_1^5 + x_2^5 &\ge y_1^5 + y_2^5,
\end{aligned}
\end{equation*} 
whereas \eqref{xiyi7inf2} becomes
\begin{equation*}
x_1^7 + x_2^7 \le y_1^7 + y_2^7.
\end{equation*}
From these inequalities  and Lemma \ref{zy7lem}, it follows that $x_1 = y_1$ and $x_2 = y_2$, and that equality holds in both parts of \eqref{xisystineqinf}.
\end{proof}

 \section{Proof of Theorem \ref{mainthm}} \label{sec:proofmain}

\bigskip

We first prove part 3 of Theorem \ref{mainthm}, which is an easy consequence of the results of the preceding sections. 
 Suppose $(a,b) \in \Sigma$, and suppose that \eqref{abineq2} holds.  Assume for contradiction that there exists
 a minimizer $u\in H^2(\mathbb R)$ for $J(a,b)$.  Then by Proposition \ref{minsol},  there must 
exist real numbers $D_1, D_2, \gamma_1, \gamma_2$ with $0 \le D_1 < D_2$ such that
 $u = \psi_{D_1,D_2;\gamma_1,\gamma_2}$.  Since $E_2(\psi)=a$ and $E_3(\psi)=b$, it follows from \eqref{Ekvalue} that
$$
\begin{aligned}
12\left(D_1^{3/2}+D_2^{3/2}\right)&=a\\
-\frac{36}{5}\left(D_1^{5/2}+D_2^{5/2}\right)&=b.
\end{aligned}
$$
Hence the equations \eqref{ABsyst} hold with $A=(a/12)^{1/3} $, $B=(-5b/36)^{1/5} $, $k=2$, $x_1=D_1^{1/2}$, and $x_2 = D_2^{1/2}$. 
 Therefore, by Lemma \ref{ABexistlem}, we must have that
$B/A \ge (1/2)^{2/15}$.  Further,  we cannot have that $B/A = (1/2)^{2/15}$, for by part 3 
of Lemma \ref{delgamlemma},  this would imply that $x_1 = x_2$, contradicting the fact that $D_1 < D_2$.
Hence $B/A > (1/2)^{2/15}$. But this means that  
\begin{equation*}
 b < -\frac{\mu a^{5/3}}{2^{2/3}},
\end{equation*} which contradicts our assumption \eqref{abineq2}.   This then completes the proof of part 3 of the Theorem.

\bigskip

Turning to the proof of parts 1 and 2 of Theorem \ref{mainthm}, we now suppose that
 $(a,b) \in \Sigma$ and either \eqref{abeq} or \eqref{abineq1} holds.  In particular we must have that $b<0$.

Let $\{\phi_n\}$ be any minimizing sequence for $J(a,b)$, so that 
$\displaystyle\lim_{n \to \infty} E_2(\phi_n) = a$, $\displaystyle\lim_{n \to \infty} E_3(\phi_n)=b$, 
and $\displaystyle\lim_{n\to \infty} E_4(\phi_n) = J(a,b)$.
(Note that minimizing sequences always exist, since $\Lambda(a,b)$ is nonempty by Proposition \ref{defsig}.)

Since $\{E_2(\phi_n)\}$ converges, then $\{\phi_n\}$ is bounded in $L^2$.
Also, since by Sobolev embedding and interpolation we have
$$
\begin{aligned}
\int_{\mathbb R} (\phi_n')^2\ dx& = 2 E_3(\phi_n) + \frac13 \int_{\mathbb R} u^3 \\
&\le 2E_3(\phi_n) + C\|\phi_n\|_{H^{1/6}}^3 \le 2E_3(\phi_n) + C\|\phi_n\|_{L^2}^{5/2}\|\phi_n\|_{H^1}^{1/2},
\end{aligned}
$$
it follows that
$$
\|\phi_n\|_{H^1}^2 \le C(1+\|\phi_n\|_{H^1}^{1/2}),
$$
which implies that $\{\phi_n\}$ is bounded in $H^1$.  Finally, we have
\begin{equation}
\int_{\mathbb R}(\phi_n'')^2\ dx = 2E_4(\phi_n) +\int_{\mathbb R}\left(\frac53uu_x^2-\frac{5}{16}u^4\right)\ dx,
\label{h2bound}
\end{equation}
and since $\{\phi_n\}$ is bounded in $H^1$, it follows from Sobolev inequalities that the integral on the right is bounded. 
  Since $\{E_4(\phi_n)\}$ is bounded above and $\{\int_{\mathbb R}(\phi_n'')^2\ dx\}$ is bounded below,
 it follows from \eqref{h2bound} that both these sequences are in fact bounded. Therefore $\{\phi_n\}$
 is bounded in $H^2$, and $J(a,b) > -\infty$. 
  
Define now $\{\rho_n\}$ as in \eqref{defrhon}.  We observe that  $\{\rho_n\}$ is not a vanishing sequence in the sense of Definition \ref{defvanish}.  Indeed, if $\{\rho_n\}$ did vanish, then it would follow from Lemma \ref{lpvanish} that $\displaystyle\lim_{n \to \infty} \|\phi_n\|_{L^3} = 0$, which in turn
implies that
$$
\liminf_{n \to \infty} E_3(\phi_n) = \frac12 \liminf_{n \to \infty} \int_{\mathbb R}(\phi_n')^2 \ge 0.
$$
But on the other hand,
$$
\lim_{n \to \infty} E_3(\phi_n) = b < 0,
$$
giving a contradiction.

 Since $\{\phi_n\}$ is bounded in $H^2(\mathbb R)$, and assumption \eqref{assumenovanish} holds, then (with the attendant definitions of $w_n$, $v^i_n$, $\theta^i_n$, $g_i$, $a_i$, and $b_i$) all the conclusions of Lemmas \ref{hksplitlem2} through \ref{e4thetanlem} are valid.  We now proceed to examine the consequences of these lemmas in the case when $\{\phi_n\}$ is a minimizing sequence for $J(a,b)$.
  
First note that, for each $i \in \mathbb N$, if $g_i \equiv 0$, then obviously $a_i=b_i=0$. 
 If on the other hand $g_i$ is not identically zero, then by Proposition \ref{defsig}, $(a_i,b_i) \in \Sigma$,
 and so $J(a_i,b_i)$ is well-defined by \eqref{defJ}.  In that case, we have:

\begin{lem} For each $i \in \mathbb N$, if $g_i$ is not identically zero, then $g_i$ is a minimizer for $J(a_i,b_i)$. 
\label{gimin}
\end{lem}

\begin{proof}  We prove the lemma by contradiction.   If $g_i$ is not a minimizer for $J(a_i,b_i)$, then there must exist 
a function $h \in H^2$ such that $E_2(h)=a_i$, $E_3(h)=b_i$, and $E_4(h) < E_4(g_i)$.   
Define, for $n \in \mathbb N$,
$$
h_n(x) = h(x - x^i_n)
$$
and
$$
\tilde \phi_n = h_n+ \sum_{\substack{ j=1\\
                j \ne i}}^n v^j_n + w_n.
$$
To obtain the desired contradiction, we will show that 
\begin{equation*}
\begin{aligned}
 \lim_{n \to \infty}E_2(\tilde \phi_n)&=a\\
 \lim_{n \to \infty}E_3(\tilde\phi_n)&=b\\
  \liminf_{n \to \infty}E_4(\tilde\phi_n)&<J(a,b).
  \end{aligned}
  \end{equation*} 

To begin with, use the triangle inequality to write 
\begin{flalign*} 
&\left|E_2(\phi_n)-E_2(\tilde \phi_n)\right|  \le \left|E_2(\phi_n)-  \sum_{ j=1}^n E_2(v^j_n)-E_2(w_n)\right|+& \\
 &   \qquad         + \left|E_2(\tilde\phi_n) -E_2(h_n)-   \sum_{\substack{ j=1\\
                j \ne i}}^n E_2(v^j_n)-E_2(w_n)\right| +\left|E_2(v^i_n)-E_2(h_n)\right|.  & 
                \end{flalign*}
We can use Lemma  \ref{hksplitlem2} and Lemma \ref{hksubstlem2} with $f=h_n$ to estimate the first two terms on the right-hand side of the
preceding inequality, and thus obtain that 
                \begin{equation*}                
\left|E_2(\phi_n)-E_2(\tilde \phi_n)\right| 
\le  |E_2(v^i_n)-E_2(h_n)|+ C\sum_{ j =1}^n \int_{A^j_n} \rho_n +C\|h_n\|_{L^2(Z^i_n)}.
\end{equation*}  
 
Since 
$$
\|h_n\|_{L^2(Z^i_n)}=\left( \int_{\mathbf R\backslash B(0,r^i_n/2)}h^2(x)\ dx\right)^{1/2},
$$
and $h \in L^2(\mathbb R)$ and  $\displaystyle \lim_{n \to \infty} r^i_n = \infty$, it follows that
 $\displaystyle \lim_{n \to \infty} \|h_n\|_{L^2(Z^i_n)}=0$.   Finally, we have that
\begin{equation*}
 \lim_{n \to \infty}(E_2(v^i_n)-E_2(h_n)) = \lim_{n \to \infty}(E_2(v^i_n)-E_2(h)) = \lim_{n \to \infty}(E_2(v^i_n)-a_i)  = 0.
\end{equation*}
Combining these results, we obtain that  $\displaystyle \lim_{n \to \infty}(E_2(\phi_n)-E_2(\tilde \phi_n))=0$,
 and hence $\displaystyle \lim_{n \to \infty} E_2(\tilde \phi_n)=a$.

Similar arguments apply to $E_3(\phi_n)-E_3(\tilde \phi_n)$ and $E_4(\phi_n)-E_4(\tilde \phi_n)$.
 From Lemmas \ref{hksplitlem2} and \ref{hksubstlem2}  we obtain that
\begin{equation}
 |E_3(\phi_n)-E_3(\tilde \phi_n)|  \le  |E_3(v^i_n)-E_3(h_n)|+ 
C\sum_{ j=1}^n \int_{A^j_n} \rho_n +C\|h_n\|_{H^1(Z^i_n)} 
\label{e3diff}
\end{equation}
and 
\begin{equation} 
E_4(\phi_n)-E_4(\tilde \phi_n)  \ge E_4(v^i_n)-E_4(h_n)
-C\sum_{ j=1}^n \int_{A^j_n} \rho_n -C\|h_n\|_{H^2(Z^i_n)}.
 \label{e4diff}
\end{equation}
The same considerations as in the preceding paragraph show that it follows from \eqref{e3diff}  that
$\displaystyle \lim_{n \to \infty} E_3(\tilde \phi_n)=b$.  Also, from \eqref{e4diff} and \eqref{e4glte4v} we obtain that
$$
 \liminf_{n \to \infty}\ [E_4(\phi_n)-E_4(\tilde\phi_n)] \ge \lim_{n \to \infty}\  [E_4(v^i_n)-E_4(h_n)] \ge E_4(g_i)-E_4(h) > 0,
 $$ 
 and hence
 $$
 \limsup_{n \to \infty}E_4(\tilde \phi_n) < \lim_{n \to\infty} E_4(\phi_n) = J(a,b).
 $$ 
 In particular, it follows that there exists some sufficiently large $n$ for which $E_4(\tilde \phi_n) < J(a,b)$.  
 But since $E_2(\tilde \phi_n)=a$ and $E_3(\tilde \phi_n)=b$, this contradicts  the definition of $J(a,b)$. 
\end{proof}

 From Proposition \ref{minsol} and Lemma \ref{gimin}, we conclude that for each $i \in \mathbb N$, there exist
 $D_{1i}, D_{2i},\gamma_{1i}, \gamma_{2i} \in \mathbb R$ with $0 \le D_{1i} \le D_{2i}$ such that 
 \begin{equation}
 g_i(x)=\psi_{D_{1i}, D_{2i};\gamma_{1i},\gamma_{2i}}(x).
 \label{defD1iD2i}
 \end{equation} 
 Here we follow the conventions that if $D_{1i}=0$, then $\psi_{D_{1i},D_{2i};\gamma_{1i},\gamma_{2i}} = \psi_{D_{2i};\gamma_{1i}}$;
 and if $D_{1i}=D_{2i}=0$, then $\psi_{D_{1i},D_{2i}} \equiv 0$.
 Also, in what follows we will occasionally omit the subscripts $\gamma_{1i}$ and $\gamma_{2i}$, 
referring to $g_i$ simply as $\psi_{D_{1i},D_{2i}}$.

 \begin{lem} 
For the numbers $D_{1i}$ and $D_{2i}$ defined  for $i \in \mathbb N$ by \eqref{defD1iD2i}, we have
\begin{equation}
\begin{aligned}
 12\sum_{i=1}^\infty \left(D_{1i}^{3/2} + D_{2i}^{3/2}\right)+\frac12\limsup_{n\to \infty} \intR w_n^2  &\le a\\
\frac{36}{5}\sum_{i=1}^\infty \left(D_{1i}^{5/2} + D_{2i}^{5/2}\right)-\frac12\limsup_{n\to \infty} \intR (w_n')^2  &\ge -b\\
\frac{36}{7}\sum_{i=1}^\infty \left(D_{1i}^{7/2} + D_{2i}^{7/2}\right)+\frac12\limsup_{n\to \infty} \intR (w_n'')^2  &\le J(a,b).
\end{aligned}
\label{Dineq2}
\end{equation}  
\label{Dineqlem2}
 \end{lem}
 
 \begin{proof} For $m=2, 3, 4$, if we define $\epsilon_{mn}$ for $n \in \mathbb N$ by 
$$
\epsilon_{mn} =  E_m(\phi_n) - \sum_{i=1}^n E_m(v^i_n) - E_m(w_n),
$$
then we have from Lemmas \ref{hksplitlem2} and \ref{sumrhontozero} that $\displaystyle \lim_{n \to \infty} \epsilon_{mn} = 0$.

In case $m=2$, we have $E_2(f) \ge 0$ for all $f \in L^2$.  Therefore we have, for all $N, n \in \mathbb N$ such that $n > N$,
\begin{equation*}
\sum_{i=1}^N E_2(v^i_n) = E_2(\phi_n) - \sum_{i=N+1}^n E_2(v^i_n) - E_2(w_n) - \epsilon_{2n} \le E_2(\phi_n)-E_2(w_n) -\epsilon_{2n}. 
\end{equation*}
Holding $N$ fixed and taking  $n$ to infinity on both sides, and recalling \eqref{defaibi}, we obtain that
$$
a \ge \sum_{i=1}^N E_2(g_i) +\frac12\limsup_{n\to \infty} \intR w_n^2 = 12\sum_{i=1}^N  \left(D_{1i}^{3/2} + D_{2i}^{3/2}\right)+ \frac12\limsup_{n\to \infty} \intR w_n^2.
$$
Then taking the limit as $N \to \infty$ yields the first inequality in \eqref{Dineq2}.

Next, we consider the case $m=3$.  We have, for all $N, n \in \mathbb N$ such that $n > N$,
\begin{equation}
\begin{aligned} 
\sum_{i=1}^N& E_3(v^i_n) = E_3(\phi_n) - \sum_{i=N+1}^n E_3(v^i_n) - E_3(w_n) - \epsilon_{3n} \\
& = E_3(\phi_n) - \sum_{i=N+1}^n \intR \left(\frac12 (v^i_n{}')^2 - \frac16(v^i_n)^3\right)- \intR \left(\frac12 (w_n')^2 -\frac16w_n^3\right) - \epsilon_{3n}\\
& \le E_3(\phi_n) + \frac16\sum_{i=N+1}^n \intR (v^i_n)^3 -\frac12\intR (w_n')^2 +\frac16 \intR w_n^3 -\epsilon_{3n}.
\end{aligned}
\label{E3phincase34}
\end{equation}
  
Let $\epsilon > 0$ be given.  From \eqref{limlimsupf3} it follows that there exists $N \in \mathbb N$ such that 
$$
\limsup_{n \to \infty} \sum_{i=N+1}^n \intR \left|v_n^i\right|^3 < \epsilon.
$$ 
For this fixed value of $N$, by taking  $n$ to infinity of both sides of \eqref{E3phincase34} and using Lemma \ref{lplimit2}, we obtain that
$$
-\frac{36}{5}\sum_{i=1}^N \left(D_{1i}^{5/2}+D_{2i}^{5/2}\right) +\limsup_{n\to \infty}\frac12\intR(w_n')^2 \le b + \frac{\epsilon}{6},
$$
and hence
\begin{equation*}
\begin{aligned}
\frac{36}{5}\sum_{i=1}^\infty \left(D_{1i}^{5/2}+D_{2i}^{5/2}\right)&-\limsup_{n\to \infty}\frac12\intR(w_n')^2  \ge\\
\frac{36}{5}&\sum_{i=1}^N \left(D_{1i}^{5/2}+D_{2i}^{5/2}\right)
-\limsup_{n\to \infty}\frac12\intR(w_n')^2  \ge -b -\frac{\epsilon}{6}.
\end{aligned}
\end{equation*} 
Since this inequality holds for all $\epsilon > 0$, we have proved the second inequality in \eqref{Dineq2}.

In case $m=4$, we have
\begin{equation}
\begin{aligned}
\sum_{i=1}^N E_4(v^i_n) +\frac12\intR(w_n'')^2 = E_4(\phi_n) &+ \sum_{i=N+1}^n \intR \left[\frac56 \left|v_n^i({v_n^i}')^2\right|  +\frac{5}{32} \left|v_n^i\right|^4\right] \\
&+\frac56 \intR w_n(w_n')^2 -\frac{5}{32}\intR w_n^4- \epsilon_{4n},
\end{aligned}
\label{E4newest}
\end{equation}
for all $N, n \in \mathbb N$ such that $n > N$.  Using Lemma \ref{lplimit2} we see that $\displaystyle \lim_{n \to \infty} \intR w_n^4 =0$ and
$$ 
\lim_{n \to \infty}\left| \intR w_n(w_n')^2 \right| \le \lim_{n \to \infty}\|w_n\|_{L^\infty}\|w_n\|_{H^1}^2 = 0. 
$$ 

For given $\epsilon > 0$, by \eqref{limlimsupf4}, we can choose $N_0 \in \mathbb N$ such that for all $N \ge N_0$,
$$
\limsup_{n \to \infty} \sum_{i=N+1}^n \intR \left[\frac56 \left|v_n^i({v_n^i}')^2\right|  +\frac{5}{32} \left|v_n^i\right|^4\right] < \epsilon.
$$ 
For each fixed value of $N \ge N_0$, taking $n$ to infinity on both sides of \eqref{E4newest}
 and using \eqref{e4glte4v}, we then obtain
$$
\frac{36}{7}\sum_{i=1}^N \left(D_{1i}^{7/2}+D_{2i}^{7/2}\right)+\limsup_{n \to \infty}\intR(w_n'')^2 \le \sum_{i=1}^N E_4(g_i) \le J(a,b) +\epsilon.
$$
Since this is true for all $N \ge N_0$, it follows that
$$
\frac{36}{7}\sum_{i=1}^\infty \left(D_{1i}^{7/2}+D_{2i}^{7/2}\right) +\limsup_{n \to \infty}\intR(w_n'')^2  \le J(a,b) + \epsilon,
$$
and since $\epsilon > 0$ was arbitrary, this proves the final inequality in \eqref{Dineq2}.
\end{proof}

By Lemma \ref{Dineqlem2},
 only finitely many of the numbers $D_{1i}$ and $D_{2i}$ can be greater than any fixed positive number.  Therefore it is possible to 
re-order the  numbers in the sequence 
 \begin{equation}
 \left(D_{11}^{1/2}, D_{21}^{1/2},D_{12}^{1/2},D_{22}^{1/2}, D_{13}^{1/2},D_{23}^{1/2}\dots\right)
 \label{Dseqk}
 \end{equation} 
 so that they form a non-increasing sequence, whose terms we denote by $\{x_n\}$, with $x_1 \ge x_2 \ge x_3 \ge \dots$.  
 
 We have the following corollary of Lemma \ref{Dineqlem2}.
 
 \begin{lem}
 Suppose that there exist numbers $y_1$ and $y_2$ with $y_1 \ge y_2 \ge 0$ and $y_1>0$ such that
 \begin{equation}
 \begin{aligned}
y_1^3+y_2^3 &= \left(\frac{1}{12}\right)a\\
y_1^5+y_2^5 &= \left(\frac{-5}{36}\right)b\\
y_1^7+y_2^7 &\ge \left(\frac{7}{36}\right)J(a,b).
\end{aligned}
\label{ysandab}
\end{equation}
  If
 \begin{equation}
 \begin{aligned}
 \sum_{i=1}^\infty x_i^3 &\le a\\
 \sum_{i=1}^\infty x_i^5 &\ge -b\\
 \sum_{i=1}^\infty x_i^7 &\le J(a,b),
 \end{aligned}
 \end{equation}
then $x_1=y_1$, $x_2=y_2$, and $x_i = 0$ for $i \ge 3$.  Moreover, 
 \begin{equation}
 \lim_{n \to \infty}\|w_n\|_{H^2}= 0.
 \label{newpfwh2}
 \end{equation}
 \label{newreducelem} 
 \end{lem}
 
 \begin{proof} This follows immediately from Lemma \ref{Dineqlem2} and Lemma \ref{ruleoutthreeposinf}.
 \end{proof}

\bigskip

 {\bf Proof of part 1 of Theorem \ref{mainthm}}.  
Suppose that \eqref{abeq} holds. We let $C = (a/12)^{2/3}= (-5b/36)^{2/5}>0$.  For every $\gamma \in \mathbb R$,  we have 
 \begin{equation}
 \begin{aligned}
 E_2(\psi_{C,\gamma})&=12C^{3/2}=a,\\
 E_3(\psi_{C,\gamma})&=-(36/5)C^{5/2}=b.
 \label{e23c}
 \end{aligned}
 \end{equation}
 From the definition of $J(a,b)$ it therefore follows that
 \begin{equation}
 J(a,b) \le E_4(\psi_{C,\gamma})=E_4(C)=(36/7)C^{7/2}. 
 \label{jabc}
 \end{equation} 

  Let $y_1=C^{1/2}$ and $y_2=0$.  Then \eqref{ysandab} is satisfied, so by Lemma \ref{newreducelem},  we must have that 
$x_1=y_1$, $x_2=0$, and \eqref{newpfwh2} holds.  It follows that  $g_1 =  \psi_{C,\gamma}$ for some $\gamma \in \mathbb R$, and  $g_i \equiv 0$ and 
$a_i = b_i = 0$ for all $i \ge 2$.
   
  We therefore have that 
\begin{equation*}
  \begin{aligned}
   \lim_{n \to \infty} E_2(\phi_n)&=a=12 C^{3/2}=E_2(g_1)\\
    \lim_{n \to \infty} E_3(\phi_n)&=b=-\frac{36}{5} C^{5/2}=E_3(g_1).
    \end{aligned}
    \end{equation*}
    Also,   from \eqref{Dineq2} we have that
   $(36/7) C^{7/2} \le J(a,b)$, and combined with \eqref{jabc}, this gives
\begin{equation}
  \lim_{n \to \infty} E_4(\phi_n)=J(a,b)=\frac{36}{7}C^{7/2}=E_4(g_1).
  \label{lime4phi}
\end{equation} 
In particular, we have now shown that $g_1$, and hence also every element of $S(C)$, is a minimizer for $J(a,b)$.

From Lemmas \ref{hksplitlem2} and \ref{sumrhontozero}, we have  that, for $m=2,3,4$,
\begin{equation*}
E_m(\phi_n) = E_m(v_n^1) + \sum_{i=2}^n E_m(v_n^i) + E_m(w_n) + \epsilon_n^m,
\end{equation*}
where $\displaystyle \lim_{n \to \infty} \epsilon_n^m = 0$.  From Lemma \ref{giconvstronginf} we conclude that
\begin{equation*}
\begin{aligned}
\lim_{n \to \infty} E_2(v_n^1)&=\lim_{n \to \infty} E_2(\theta_n^1)=E_2(g_1)= a = \lim_{n \to \infty} E_2(\phi_n)\\ 
\lim_{n \to \infty} E_3(v_n^1)&=\lim_{n \to \infty} E_3(\theta_n^1)=E_3(g_1) = b = \lim_{n \to \infty} E_3(\phi_n).
\end{aligned}
\end{equation*}
  Therefore for $m=2$ and $m=3$ we have
\begin{equation}
\lim_{n \to \infty} \left[E_m(w_n) + \sum_{i=2}^n E_m(v_n^i) \right] = 0.
\label{leftovers}
\end{equation}
When $m=2$, \eqref{leftovers} immediately implies that
\begin{equation}
\lim_{n \to \infty}  \sum_{i=2}^n \|v_n^i\|^2_{L^2(\mathbb R)} = 0.
\label{leftoverL2}
\end{equation}

We claim that when $m=3$, \eqref{leftovers} implies that
\begin{equation}
\lim_{n \to \infty} \sum_{i=2}^n \|(v_n^i)'\|^2_{L^2(\mathbb R)} = 0.
\label{leftoverH1}\end{equation}
To prove this, since $\displaystyle \lim_{n \to \infty}E_3(w_n) = 0$ by \eqref{newpfwh2} and Lemma \ref{lplimit2}, it is enough to show that
\begin{equation}
\lim_{n \to \infty} \sum_{i=2}^n \intR |v_n^i|^3 = 0.
\label{clumpofvs}
\end{equation}
Let $\epsilon > 0$ be given.   By \eqref{limlimsupf3}, we can choose
$N_1$ such that
$$
\limsup_{n \to \infty} \sum_{i=N_1}^n \intR |v_n^i|^3 < \epsilon.
$$ 
  Therefore, there exists $N_2$ such that for all $n \ge N_2$,
$$
 \sum_{i=N_1}^n \intR |v_n^i|^3 < \epsilon.
$$
For each fixed $i \ge 2$, since $g_i \equiv 0$, it follows from Lemma \ref{giconvstronginf} that 
$\displaystyle \lim_{n \to \infty} \|v_n^i\|_{H^1(\mathbb R)}=0$,
 and hence by Sobolev embedding that $\displaystyle \lim_{n \to \infty} \|v_n^i\|_{L^p(\mathbb R)}=0$
 for all $p \ge 2$. So there exists $N_3$ such that for all $n \ge N_3$,
$$
 \sum_{i=2}^{N_1-1} \intR |v_n^i|^3 < \epsilon.
$$
Then for all $n \ge \max(N_2,N_3)$,
$$
 \sum_{i=2}^n \intR |v_n^i|^3 < 2\epsilon,
$$
proving \eqref{clumpofvs} and \eqref{leftoverH1}.

  Define $\tilde \phi_n(x):=\phi_n(x+x_n^1)$ for $n \in \mathbb N$. From \eqref{phidecomp} and \eqref{defthetain} we have
\begin{equation*}
\tilde \phi_n = \theta_n^1 + \sum_{i=2}^n \tilde v_n^i + \tilde w_n
\end{equation*}
for all $n \in \mathbb N$, where $\tilde v_n^i(x):=v_n^i(x+x_n^1)$ and $\tilde w_n(x):=w_n(x+x_n^1)$.  Therefore
\begin{equation}
\|\tilde \phi_n - g_1\|_{H^1(\mathbb R)} \le \|\theta_n^1 - g_1\|_{H^1(\mathbb R)} +
\left\| \sum_{i=2}^n\tilde v_n^i\right\|_{H^1(\mathbb R)} + \|\tilde w_n\|_{H^1(\mathbb R)}.
\label{forH1convpart1}
\end{equation}
Now since, for a given $n \in \mathbb N$, the supports of $\{\tilde v_n^i\}_{i =1, \dots, n}$ are mutually disjoint,  we have that
$$
\left\| \sum_{i=2}^n \tilde v_n^i\right\|_{H^1(\mathbb R)}^2 = \sum_{i=2}^n\|\tilde v_n^i\|^2_{H^1(\mathbb R)}.
$$
Hence, from Lemma \ref{giconvstronginf}, \eqref{leftoverL2}, \eqref{leftoverH1}, and \eqref{forH1convpart1} we conclude that $\tilde \phi_n(x)$ 
converges strongly to $g_1$ in $H^1(\mathbb R)$.   
 
In particular, since $\{\tilde \phi_n\}$ is bounded in $H^2(\mathbb R)$, it follows by the same arguments
 used to prove Lemma \ref{vnitogilowertermslem} that 
 $$
 \lim_{n \to \infty} \int_{\mathbb R} \tilde \phi_n(\tilde \phi_n')^2 = \int_{\mathbb R} g_1(g_1')^2
 $$
 and
 $$
 \lim_{n \to \infty} \int_{\mathbb R} \tilde \phi_n^4 = \int_{\mathbb R} g_1^4.
 $$
 Therefore
 \begin{equation*}
 \begin{aligned}
  \lim_{n \to \infty}\int_{\mathbb R} (\tilde \phi_n'')^2 &= \lim_{n \to \infty} \left(E_4(\tilde \phi_n)+
\frac53 \int_{\mathbb R} \tilde \phi_n(\tilde \phi_n')^2 
  -\frac{5}{16}\int_{\mathbb R} \tilde \phi_n^4\right)\\
&  =E_4(g_1) +\frac53\int_{\mathbb R} g_1(g_1')^2-\frac{5}{16} \int_{\mathbb R} g_1^4
 = \int_{\mathbb R} (g_1'')^2. 
 \end{aligned}
 \end{equation*}
 Hence we have that
 \begin{equation}
 \lim_{n \to \infty} \|\tilde \phi_n\|_{H^2(\mathbb R)} = \|g_1\|_{H^2(\mathbb R)}.
 \label{convH2norm}
 \end{equation}
 But, from the weak compactness of the unit sphere in Hilbert space, we may assume 
 by passing to a further subsequence that $\{\tilde \phi_n\}$ converges weakly in $H^2(\mathbb R)$,
 and the limit must be $g_1$.   From \eqref{convH2norm} it then follows that $\{\tilde \phi_n\}$
 must converge strongly to $g_1$ in $H^2(\mathbb R)$.  
  This implies that 
 $$
 \lim_{n \to \infty} \|\phi_n - \psi_{C,\gamma +x_n^1}\|_{H^2(\mathbb R)} =0,
 $$
 which, since $\psi_{C,\gamma+x_n^1} \in S(C)$ for all $n \in \mathbb N$, shows that $\{\phi_n\}$ converges strongly
 to $S(C)$ in $H^2(\mathbb R)$.  This then completes the proof of part 1 of Theorem \ref{mainthm}.

\bigskip

{\bf Proof of part 2 of Theorem \ref{mainthm}.} 
 Assume that \eqref{abineq1} holds.   
 Applying part 2 of Lemma \ref{delgamlemma} with $A= (a/12)^{1/3}$ and $B=(-5b/36)^{1/5}$, we obtain that there exists
 a unique pair of numbers $y_1$ and $y_2$ such that $0 < y_2 < y_1$
 and \eqref{alphbetequation} holds.   Define $C_1=y_2^2$ and $C_2 = y_1^2$; then we have $0 < C_1 < C_2$ and
 \begin{equation}
 \begin{aligned}
 E_2(C_1,C_2)&=12 \left( C_1^{3/2} + C_2^{3/2}       \right) = a\\
E_3(C_1,C_2)&=\frac{-36}{5}\left(C_1^{5/2} + C_2^{5/2}    \right) = b.
 \end{aligned}
 \label{e23c1c2}
 \end{equation} 
Therefore,   for every pair $(\gamma_1,\gamma_2) \in \mathbb R^2$, we have $E_2(\psi_{C_1,C_2;\gamma_1,\gamma_2})=a$ and 
$E_3(\psi_{C_1,C_2;\gamma_1,\gamma_2})=b$; and hence from the definition of $J(a,b)$ we have that 
 \begin{equation}
 E_4(\psi_{C_1,C_2;\gamma_1,\gamma_2})=E_4(C_1,C_2) = \frac{36}{7} \left( C_1^{7/2}+C_2^{7/2}     \right) \ge J(a,b).
 \label{e4c1c2}
 \end{equation}

 From \eqref{e23c1c2}, and \eqref{e4c1c2} it follows that \eqref{ysandab} holds. Hence Lemma \ref{newreducelem} yields that  $x_1=y_1$, $x_2 = y_2$, $x_i =  0$ for all $i \ge 3$, and \eqref{newpfwh2} holds.  

We thus see that (after relabelling the numbers $D_{1i}$ and $D_{2i}$ if necessary), we can reduce consideration 
to two possible cases:  Case I in which 
$$0<D_{11}=C_1 < D_{21}=C_2,$$
 and $D_{1i}=D_{2i}=0$ for all $i \ge 2$, and Case II in which
$$
0=D_{11}< D_{21}=C_1, \quad 0 = D_{12} < D_{22}=C_2, 
$$ 
and $D_{1i}=D_{2i}=0$ for all $i \ge 3$. 

In Case I, we have that $g_1 = \psi_{C_1,C_2,\gamma_1,\gamma_2}$ for some $(\gamma_1,\gamma_2) \in \mathbb R^2$.   Then
\begin{equation*}
\begin{aligned}
  \lim_{n \to \infty} E_2(\phi_n)&=a=E_2(g_1)\\
    \lim_{n \to \infty} E_3(\phi_n)&=b=E_3(g_1), 
    \end{aligned}
\end{equation*}
and from \eqref{Dineq2} and \eqref{e4c1c2} we have that
\begin{equation*}
  \lim_{n \to \infty} E_4(\phi_n)=J(a,b)=E_4(g_1). 
\end{equation*} 
  In particular, this implies that $g_1$, along with every other element of $S(C_1,C_2)$, is a minimizer for $J(a,b)$.
  
  The same argument as in the paragraphs following equation \eqref{lime4phi} now shows that the translated sequence $\tilde \phi_n(x)=\phi_n(x+x_n^1)$
  converges strongly in $H^2(\mathbb R)$ to $g_1$.  Hence
  $$
 \lim_{n \to \infty} \|\phi_n - \psi_{C_1,C_2,\gamma_1 +x_n^1,\gamma_2 + x_n^1}\|_{H^2(\mathbb R)} =0,
 $$
 which shows that $\{\phi_n\}$ converges strongly to $S(C_1,C_2)$ in $H^2(\mathbb R)$. This  completes the proof 
of part 2 of Theorem \ref{mainthm} in Case I.
 
 We turn now to Case II.  In this case, we have that $g_1 = \psi_{C_1,\gamma_1}$ and $g_2 = \psi_{C_2,\gamma_2}$
 for some $(\gamma_1,\gamma_2) \in \mathbb R^2$; and
$g_i \equiv 0$ for all $i \ge 3$.    Then from \eqref{e23c1c2} we have that
\begin{equation*}
\begin{aligned}
  \lim_{n \to \infty} E_2(\phi_n)&=a=E_2(g_1)+E_2(g_2)\\
    \lim_{n \to \infty} E_3(\phi_n)&=b=E_3(g_1)+E_3(g_2), 
    \end{aligned}
\end{equation*}
and from \eqref{Dineq2} and \eqref{e4c1c2} we have that
\begin{equation}
  \lim_{n \to \infty} E_4(\phi_n)=J(a,b)=E_4(g_1)+E_4(g_2). 
  \label{e4phine4gn}
\end{equation} 
  In particular, this implies again that every element of $S(C_1,C_2)$ is a minimizer for $J(a,b)$.  However, now it is no longer the case that
  one can translate the functions in the sequence $\{\phi_n\}$ to obtain a strongly convergent sequence in $H^2(\mathbb R)$.
   Instead, we must modify the argument in the proof of part 1 of the Theorem, as follows.  

Repeating the argument used above to obtain \eqref{leftovers}, we obtain in this case that
\begin{equation*}
\lim_{n \to \infty} \left[    E_m(w_n)+\sum_{i=3}^n E_m(v^i_n)  \right]=0
 \end{equation*}
for $m=2$ and $m=3$.   Also, since 
$$
\intR \left| w_n(w_n')^2\right| \le \|w_n\|_{L^\infty(\mathbb R)} \intR \left(w_n' \right)^2 \le C \|w_n\|_{H^1(\mathbb R)}^3
$$ 
by the Sobolev Embedding Theorem, it follows from \eqref{newpfwh2} and Lemma \ref{lplimit2} that 
  \begin{equation}
 \lim_{n \to \infty} E_m(w_n) = 0.
 \label{e4wnp}
  \end{equation}
  for $m=2,3,4$.  Therefore
  \begin{equation}
\lim_{n \to \infty} \sum_{i=3}^n E_m(v^i_n) =0
\label{leftoverscaseIVb}
\end{equation}
for $m=2,3,4$.

 When $m=2$, \eqref{leftoverscaseIVb} immediately implies that
\begin{equation*}
\lim_{n \to \infty}\sum_{i=3}^n \|v_n^i\|^2_{L^2(\mathbb R)} = 0.
\end{equation*}
Also, by the same proof used above to prove \eqref{clumpofvs}, we have in this case that
\begin{equation}
\lim_{n \to \infty} \sum_{i=3}^n \intR |v_n^i|^3 = 0,
\label{clumpofvs2}
\end{equation}
and together with \eqref{leftoverscaseIVb} for $m=3$, this implies that
 \begin{equation}
\lim_{n \to \infty} \sum_{i=3}^n \|v_n^i\|^2_{H^1(\mathbb R)} = 0.
\label{wnconvH1again}
\end{equation}
 
From \eqref{hksplit2}, \eqref{e4phine4gn}, and \eqref{e4wnp}, we have
$$
\lim_{n \to \infty} E_4(\phi_n) =\lim_{n \to \infty} \sum_{i=1}^n E_4(v_n^i) + E_4(w_n)   = E(g_1)+E(g_2). 
$$   
Also, by the same argument used to deduce \eqref{clumpofvs} and \eqref{clumpofvs2} from \eqref{limlimsupf3},
 it follows from \eqref{limlimsupf4} that
\begin{equation}
\lim_{n \to \infty} \sum_{i=3}^n \intR \left[\left|v_n^i \left({v_n^i}'\right)^2\right| + |v_n^i|^4\right] = 0.
\label{clumpofvsL4}
\end{equation}  Therefore we can write
\begin{equation}
\lim_{n \to \infty} E_4(\phi_n)
=\lim_{n \to \infty}\left[ E_4(v_n^1) + E_4(v_n^2) +  \frac12 \sum_{i=3}^n \intR \left(v_{n}^{i\ \prime \prime}\right)^2  \right].
\label{e4decompcase4}
\end{equation}
For every sequence $\{n_k\}_{k \in \mathbb N}$ of integers approaching infinity, it follows from \eqref{clumpofvsL4}, \eqref{e4decompcase4},
  and Fatou's Lemma that 
\begin{equation}
\begin{aligned}
\lim_{k \to \infty} E_4(\phi_{n_k})&\ge  \liminf_{k \to \infty} E_4(v_{n_k}^1) +  \liminf_{k \to \infty} E_4(v_{n_k}^2)
+\frac12 \sum_{i =3}^\infty \liminf_{k \to \infty}\intR \left(v_{n_k}^{i\ \prime \prime}\right)^2  \\
 &=  \liminf_{k \to \infty} E_4(v_{n_k}^1) +  \liminf_{k \to \infty} E_4(v_{n_k}^2)
+ \sum_{i =3}^\infty \liminf_{k \to \infty} E_4(v_{n_k}^i).
 \end{aligned}
\label{e4plusepscase4}
\end{equation} 

We claim now that, for every $i \in \mathbb N$,
\begin{equation}
\lim_{n \to \infty} E_4(v_n^i) = E_4(g_i).
\label{lime4vnicase4}
\end{equation}
For otherwise, we would have $\displaystyle \limsup_{n \to \infty} E_4(v_n^{i_0}) \ge E_4(g_{i_0}) + \epsilon$ for some $i_0 \in \mathbb N$
 and some $\epsilon > 0$. In that case it would then follow from Lemma \ref{e4thetanlem} and \eqref{e4plusepscase4} 
that there exists a sequence of integers $\{n_k\}_{k \in \mathbb N}$
 approaching infinity for which
$$ 
\lim_{k \to \infty} E_4(\phi_{n_k}) \ge \sum_{i=1}^\infty E_4(g_i) + \epsilon.
$$
But then from \eqref{e4phine4gn}, since $g_i \equiv 0$ for $i \ge 3$, we obtain that
$$
\lim_{k \to \infty} E_4(\phi_{n_k}) \ge \lim_{n \to \infty} E_4(\phi_n) + \epsilon.
$$
This contradiction proves our claim.
 
 For all $i \in \mathbb N$,  since $\{\theta_n^i\}_{n \in \mathbb N}$ converges to $g_i$ strongly
 in $H^1(\mathbb R)$ and weakly in $H^2(\mathbb R)$, and since \eqref{lime4vnicase4} implies that
 $\displaystyle \lim_{n \to \infty}E_4(\theta_n^i) = E_4(g_i)$ as well, it follows from the same argument 
used to prove \eqref{convH2norm} that
 \begin{equation*}
 \lim_{n \to \infty}\|\theta_n^i -g_i\|_{H^2(\mathbb R)} = 0.
 \end{equation*}
Therefore,  we have
 \begin{equation}
 \lim_{n \to \infty}\|v_n^i - \psi_{C_i,\gamma_i + x_n^i}\|_{H^2(\mathbb R)} = 0 \quad \text{for $i = 1,2$};
 \label{vniconvH2i12}
 \end{equation}
and
\begin{equation}
\lim_{n \to \infty} \|v_n^i\|_{H^2(\mathbb R)} = 0 \quad \text{for $i \ge 3$}.
\label{vnitozeroigt3}
\end{equation}

Now for every $n \in \mathbb N$, recalling that the supports of $\{\tilde v_n^i\}_{i=1,\dots,n}$ are mutually disjoint, we have that
\begin{equation}
  \left\| \sum_{i=3}^n \tilde v_n^i\right\|_{H^2(\mathbb R)}^2=\sum_{i=3}^n\|\tilde v_n^i\|^2_{H^2(\mathbb R)}.
  \label{twosums}
\end{equation}
But from \eqref{e4phine4gn}, \eqref{wnconvH1again}, \eqref{e4decompcase4}, and \eqref{lime4vnicase4}, it follows that the right-hand side of the preceding equation is uniformly bounded for $n \in \mathbb N$.
 Therefore it follows from \eqref{vnitozeroigt3}, \eqref{twosums}, and the Dominated Convergence Theorem that  
\begin{equation}
\lim_{n \to \infty}\left\| \sum_{i=3}^n \tilde v_n^i\right\|_{H^2(\mathbb R)} = 0.
\label{h2normofsum}
\end{equation}
   
 From Corollary \ref{combinecases}, we see that $\displaystyle \lim_{n \to \infty} |x_n^1 - x_n^2| = \infty$,
 since $|x_n^1-x_n^2| \ge r_n^1 + r_n^2$ and $\displaystyle \lim_{n \to \infty} r_n^1 = \lim_{n \to \infty} r_n^2= \infty$.  
From \eqref{phidecomp} and the triangle inequality, we have
 \begin{equation*}
 \begin{aligned} \|&\phi_n - \psi_{C_1,C_2,\gamma_1 + x_n^1,\gamma_2 + x_n^2}\|_{H^2(\mathbb R)} \le 
 \|v_n^1-\psi_{C_1,\gamma_1+x_n^1}\|_{H^2}+\|v_n^2-\psi_{C_2,\gamma_2+x_n^2}\|_{H^2}\\
&
+ \left\|  \psi_{C_1,\gamma_1 + x_n^1} + \psi_{C_2,\gamma_2 + x_n^2} - \psi_{C_1,C_2,\gamma_1+x_n^1,\gamma_2 + x_n^2}
 \right\|_{H^2}+\left\|\sum_{i=3}^n v_n^i\right\|_{H^2} + \|w_n\|_{H^2}.
 \end{aligned} 
\end{equation*}
But by  Lemma \ref{sumoftwosols}, \eqref{newpfwh2},  \eqref{vniconvH2i12}, and \eqref{h2normofsum}, 
all the terms on the right-hand side of the preceding inequality have limit zero as $n$ goes to infinity. 
 This then completes the proof of part 2 of Theorem \ref{mainthm}.
 
 \bigskip

{\bf Proof of Corollary \ref{stability}.} Corollary \ref{stability} follows from Theorem \ref{mainthm} by a standard argument,
 which we include here for the reader's convenience. Suppose $C_1$ and $C_2$ are given such that $0 < C_1 < C_2$, and define
$$
\begin{aligned}
a=E_2(C_1,C_2)& = 12(C_1^{3/2} + C_2^{3/2})\\
b= E_3(C_1,C_2) &= \frac{-36}{5}(C_1^{5/2} + C_2^{5/2}).\\
\end{aligned}
$$ 
Then from Lemma \ref{delgamlemma} it follows that $a$ and $b$ satisfy \eqref{abineq1}, and so the assertion about convergence of minimizing sequences to $S(C_1,C_2)$ follows from Theorem \ref{mainthm}.  

To prove stability, we argue by contradiction:  if $S$ were not stable, then there would exist a sequence of initial data $\{u_{0n}\}_{n \in \mathbb N}$ in $H^2(\mathbb R)$ such that $\displaystyle \lim_{n \to \infty} d(u_{0n},S) =0$  and a number $\epsilon > 0$ and a sequence of times $\{t_n\}_{n \in \mathbb N}$ such that the solutions $u_n(x,t)$ of KdV with initial data $u_n(\cdot,0) = u_{0n}$ would satisfy
\begin{equation}
d(u(\cdot, t_n),S) \ge \epsilon
\label{contradictstability} 
\end{equation}
for all $n \in \mathbb N$.  Let $\phi_n = u(\cdot,t_n)$ for $n \in \mathbb N$.   Since $E_2$, $E_3$, and $E_4$ are continuous functionals on $H^2(\mathbb R)$, and are conserved under the time evolution of the KdV equation, we have
$$
\begin{aligned}
\lim_{n \to \infty} E_2(\phi_n) &= \lim_{n \to \infty} E_2(u_{0n}) = E_2(C_1,C_2) = a\\
\lim_{n \to \infty} E_2(\phi_n) &= \lim_{n \to \infty} E_3(u_{0n}) = E_3(C_1, C_2) = b\\
\lim_{n \to \infty} E_4(\phi_n) &= \lim_{n \to \infty} E_4(u_{0n}) = E_4(C_1,C_2)=J(a,b).
\end{aligned}
$$
Hence $\{\phi_n\}$ is a minimizing sequence for $J(a,b)$, and so must converge strongly to $S$ in $H^2(\mathbb R)$.   But this then contradicts \eqref{contradictstability}.

\section{Acknowledgements} \label{sec:ack}

The authors would like to thank Jerry Bona for introducing us to the problems considered in this paper, and more generally for his help and guidance over many years.


\begin{thebibliography}{BBBSS}

\bibitem{Al2} J. Albert,  Concentration compactness and the stability of solitary-wave solutions to nonlocal equations, in {\it Applied analysis (Baton Rouge, LA, 1996), Contemp.\ Math.}, \textbf{221}, Amer. Math. Soc., Providence, RI, 1999.

 \bibitem{Al} J. Albert, A uniqueness result for 2-soliton solutions of the KdV equation, {\it Discrete and Continuous Dynamical Systems - A} \textbf{39} (2019), 3635--3670.

 \bibitem{ABN} J. P. Albert, J. L. Bona, and N. V. Nguyen,  On the stability of KdV multi-solitons, \textit{Differential Integral Equations} \textbf{20} (2007), 841--878.

\bibitem{Be} T. B. Benjamin, The stability of solitary waves,  \textit{Proc.\ Roy.\ Soc.\ London Ser.\ A} \textbf{328} (1972), 153--183.  

\bibitem{Bo} J. Bona, On the stability theory of solitary waves, \textit{Proc.\ Roy.\ Soc.\ London Ser.\ A} \textbf{344}(1975), 363--374. 

\bibitem{BC} H. Brezis and J.-M. Coron, Convergence of solutions of $H$-systems or how to blow bubbles,
\text{Arch.\ Rational Mech.\ Anal.} \textbf{89} (1985), 21--56.

\bibitem{CL} T. Cazenave and P.-L. Lions, Orbital stability of standing waves for some nonlinear Schr\"odinger equations, \textit{Comm. Math. Phys.} \textbf{85} (1982), 549--561. 

\bibitem{Di} L. A. Dickey, \textit{Soliton equations and
   Hamiltonian systems (1st edition)}, World Scientific, 1991.

\bibitem{Di2} L. A. Dickey, \textit{Soliton equations and
   Hamiltonian systems (2nd edition)}, World Scientific, 2003.

   \bibitem{G} P. G\'erard, Description du d\'efaut de compacit\'e de l'injection de Sobolev, \textit{ESAIM Control Optim. Calc. Var.} \textbf{3} (1998), 213--233.

\bibitem{KV} R. Killip and M. Visan, Orbital stability of KdV multisolitons in $H^{-1}$, \url{arXiv:2009.06746 [math.AP]}

\bibitem{Lau} T. Laurens, Multisolitons are the unique constrained minimizers of the KdV conserved quantities, \url{arXiv:2206.09050 [math.AP]}


\bibitem{L} P. D. Lax, Integrals of nonlinear equations of
   evolution and solitary waves, \textit{Comm. Pure Appl.
   Math.} \textbf{21} (1968), 467--490.

\bibitem{L3} P. D. Lax, Periodic solutions of the KdV equation, \textit{ Comm.\ Pure Appl.\ Math.} {\textbf 28} (1975), 141--188. 

\bibitem{L2} P. D. Lax, Almost periodic solutions of the KdV equation, {\it SIAM Rev.} {\bf 18} (1976), 351--375.

\bibitem{LCW} S. Le Coz and Z. Wang, Stability of the multi-solitons of the modified Korteweg-de Vries equation, \url{arXiv:2010.00814 [math.AP]}
   
   \bibitem{Lu} D. G. Luenberger, \textit{Optimization by Vector Space Methods}, Wiley, 1969.
   
   \bibitem{LY} D. G. Luenberger and Yinyu Ye, \textit{Linear and Nonlinear Programming}, Springer, 2008 .

\bibitem{MS} J. Maddocks and R. Sachs, On the stability
   of KdV multi-solitons, \textit{Comm. Pure Appl. Math.}
   \textbf{46} (1993), 867--901.

\bibitem{M} M. Mari\c s, Profile decomposition for sequences of Borel measures, arXiv:1410.6125.

\bibitem{Mu} Y. Murata, \textit{Mathematics for Stability and Optimization of Economic Systems}, Academic Press, 1977.

\bibitem{N} S.  P.  Novikov, A periodic problem for the Korteweg-de Vries equation. I. (Russian) 
{\it Funkcional.\ Anal.\ i Prilo\v zen.} {\bf 8} (1974), 54--66. English translation available at \url{http://www.mi-ras.ru/~snovikov/43.pdf}

\bibitem{S} S. Solimini, A note on compactness-type properties with respect to Lorentz norms of
bounded subsets of a Sobolev space, {\it Ann.\ Inst.\ Henri Poincar\'e, Anal.\ Non Lin\'eaire} {\bf 12} (1995),
319--337.

\bibitem{St} M. Struwe, A global compactness result for elliptic boundary value problems involving
limiting nonlinearities, {\it Math. Z.} {\bf 187} (1984), 511--517.

\bibitem{T} C. Tintarev, {\it Concentration compactness --- functional-analytic theory of concentration phenomena}, De Gruyter, 2020. 
 
\bibitem{W} F. Warner, \textit{Foundations of Differentiable Manifolds and Lie Groups}, Scott, Foresman, 1971. 



\end{thebibliography}
\end{document}